\documentclass[11pt]{amsart}
\usepackage{amsmath,amssymb,amsthm}
\usepackage[margin=1in]{geometry}
\usepackage[hidelinks]{hyperref}
\usepackage{xcolor}

\numberwithin{equation}{section}
\newtheorem{theorem}{Theorem}[section]
\newtheorem{lemma}[theorem]{Lemma}
\newtheorem{proposition}[theorem]{Proposition}

\newtheorem{claim}[theorem]{Claim}

\title{Ergodicity of the geodesic flow on nonpositively curved surfaces}
\author{Federico Rodriguez Hertz}
\address{Department of Mathematics, 107 McAllister Building, University Park, State College, PA 16802, USA}
\email{fjr11@psu.edu}
\author{Jana Rodriguez Hertz}
\address{Department of Mathematics and International Center of Mathematics, Southern University of Science and Technology, 1088 Xueyuan Avenue, Nanshan District, 518055 Shenzhen, Guangdong, China}
\email{rhertz@sustech.edu.cn}
\author{Weisheng Wu}
\address{School of Mathematical Sciences, Xiamen University, Xiamen, 361005, P.R. China}
\email{wuweisheng@xmu.edu.cn}
\date{}
\thanks{Federico Rodriguez Hertz has been partially supported by NSF grant
DMS-2453688. Weisheng Wu is supported by grant 2022YFA1007800
of the National Key R\&D Program of China.
}
\begin{document}
\begin{abstract}
     We prove that the geodesic flow of a closed, connected surface
 of negative Euler characteristic with a $C^{3+\mathrm{Lip}}$ metric of
 nonpositive curvature is ergodic with respect to Liouville
 measure. \par
The singular set, consisting of directions whose geodesics remain in zero curvature for all time, has zero Liouville measure. Ergodicity then follows from Pesin's theorem.
\end{abstract}
\maketitle

\section{Introduction}

\begin{theorem}\label{thm:main}
Let $M$ be a closed, connected Riemannian surface of negative Euler characteristic, 
with a $C^{3+\mathrm{Lip}}$ metric, of nonpositive curvature. Then the geodesic
flow on the unit tangent bundle $SM$ is ergodic with respect to Liouville measure.
\end{theorem}

This settles the nonpositive-curvature surface problem arising from the classical work of Hopf \cite{Hopf39}, which remained open for nearly fifty years after Pesin's work \cite{Pesin77}.

The theorem concerns the natural invariant probability measure
\[
dm=\frac{d\bar a\,d\theta}{2\pi},
\]
where $\bar a$ is the normalized Riemannian area and $\theta$ is the angular
coordinate in each unit tangent circle. Ergodicity means that every
measurable invariant subset of $SM$ has measure zero or one.

For negative curvature, this is a classical theorem. Hedlund established
ergodicity in constant negative curvature \cite{Hedlund34}, Hopf developed
the argument for surfaces of variable negative curvature
\cite{Hopf36,Hopf39}. Anosov announced the ergodicity of geodesic flows on closed negatively curved manifolds in arbitrary dimension in 1963 \cite{Anosov63} and provided a complete proof, including the Kolmogorov property, in his celebrated monograph~\cite{Anosov67}. Anosov and Sinai ~\cite{AnosovSinai67} presented a geometric proof of absolute continuity of the stable and unstable foliations, providing an alternative to Anosov’s original approach through smooth approximation. 
Their argument, now called the Hopf argument, compares forward and backward time
averages along stable and unstable leaves respectively. Contraction along these leaves,
together with absolute continuity, shows that these averages are constant
almost everywhere.

Allowing the Gaussian curvature $K$ to vanish introduces a different difficulty. On a closed
nonpositively curved surface, Eberlein's criterion says that the geodesic
flow is Anosov precisely when there is no complete geodesic along which
curvature vanishes identically \cite{Eberlein73}. Thus zeros of curvature
at individual points do not by themselves destroy uniform hyperbolicity;
the obstruction is a complete zero-curvature geodesic. The relevant set is
\[
\mathrm{Sing}
=\{v\in SM:K(\pi g^sv)=0\text{ for every }s\in\mathbb R\},
\]
where $\pi:SM\to M$ is the bundle projection and $\{g^t\}_{t\in \mathbb{R}}$ denotes the geodesic flow. Its complement
$\mathcal R=SM\setminus\mathrm{Sing}$ is the \emph{rank-one set}, also
called the \emph{regular set}. On a surface, it consists precisely of the
vectors whose geodesics encounter negative curvature. This includes geodesics whose curvature vanishes on an
entire forward or backward ray but is negative at some
other time. The tangent vectors to such geodesics form
a set of zero Liouville measure. Indeed, recurrence in
the direction of the flat ray would force the geodesic
to encounter negative curvature again, contradicting
its eventual flatness. The measure assertion therefore
follows from Poincar\'e recurrence.

Pesin's ergodicity theorem \cite{Pesin77} supplies the foundation for the
present proof: the geodesic flow restricted to $\mathcal R$ is ergodic
with respect to the restricted Liouville measure; see also the formulation in
\cite[Section~6.1]{BurnsMatveev21}. To obtain ergodicity on all of $SM$, it therefore remains to prove
\[
m(\mathrm{Sing})=0.
\] 
As a consequence, the geodesic flow is in fact Bernoulli with respect to Liouville measure. This follows from Pesin's argument
\cite[Theorem~3.1 and the proof of Theorem~9.1]{Pesin77},
since the regular set has full measure.

Ballmann--Brin and Burns \cite{BB82,Burns83} developed the
ergodic theory of geodesic flows in higher dimensions and
under weaker curvature assumptions. The geometric structure
of nonpositively curved manifolds was developed further in
\cite{BBE85I,BBS85II}. Uniqueness of the measure of maximal
entropy was studied in \cite{Knieper98,CKW21}. Burns and Katok \cite{BK85} surveyed the geometry and dynamics of manifolds of nonpositive curvature, including the problem of ergodicity of geodesic flows on rank-one manifolds.

A stronger geometric problem, posed by A. Katok, asks whether every complete zero-curvature geodesic is closed and only finitely many free homotopy classes occur; see \cite[Conjecture 1.2 and Theorem 1.3]{Wu15}. We thank Ya. B. Pesin for communicating this problem to us. Recently, Qiu~\cite{Qiu2026} constructed a $C^{3+1/3}$ metric of nonpositive curvature admitting a nonclosed flat geodesic asymptotic to a periodic one. Ruggiero's infinite-area theorem for ideal triangles with a flat side \cite{Ruggiero97} rules out such examples for smooth metrics, Petrunin \cite{Petrunin22} also gave a direct proof of this consequence. The importance of regularity was already illustrated by Contreras and Ruggiero \cite{CR97}, who constructed $C^3$ metrics for which all ideal triangles have uniformly bounded area.  Katok's question remains open at higher regularity.

The geometry of flat strips was studied by F. Rodriguez Hertz \cite{FRH03} and Cao and Xavier \cite{CaoXavier08}. Coudène and Schapira investigated their effect on approximation by periodic orbits and the structure of invariant measures. They constructed compact rank-one surfaces on which ergodic invariant measures are not dense among all invariant probability measures \cite{CS11}. They later proved that, for a compact nonpositively curved surface, ergodicity is a generic property in the space of invariant probability measures precisely when the universal cover contains no flat strip \cite{CS14}.

Weisheng Wu proved ergodicity for nonpositively curved higher genus surfaces when $\{K<0\}$ has finitely many connected components, he additionally proved periodicity and finiteness results for flat geodesics \cite{Wu15}; an unpublished proof of this had also been obtained by M. Brin, D. Burago, and S. Ivanov \cite{BBI11}. Wu, Liu, and Wang extended both the ergodicity criterion and the periodicity and finiteness conclusions to closed surfaces of genus at least two without focal points, assuming that $\{K<0\}$ has finitely many connected components \cite{WLW23}.

 Fei Liu and Xiaokai Liu give an ergodicity criterion in terms of the infinite-order vanishing set of curvature, which singles out points where curvature
vanishes to every order \cite{LiuLiu26}. Their criterion yields ergodicity when this set
has zero area, and their paper also establishes ergodicity for
real-analytic metrics.

Our proof uses angular raising and lowering operators, whose origins
lie in the representation theory of Lie groups. In his study of
unitary representations of the Lorentz group $\mathrm{SO}(2,1)$, Bargmann constructed
a pair of operators that raise and lower the weight of a vector
under rotations by one~\cite[p.~603]{Bargmann47}. Gel'fand and
Shapiro obtained the analogous formulas for representations of
$\mathrm{SO}(3)$~\cite[Section~2]{GS52}. In the setting of closed
surfaces of constant negative curvature, Gel'fand and Fomin used
raising and lowering operators on the unit tangent bundle to determine
the spectral multiplicity of the geodesic flow~\cite[Section~3]{GF52}.
Flaminio and Forni used these operators on the same surfaces to study
the horocycle flow. The symmetries of the hyperbolic plane allowed them
to work one irreducible representation at a time~\cite[Sections~2.4 and~3.1]{FF03}.

Guillemin and Kazhdan gave an intrinsic formulation of these operators
for surfaces of variable curvature, denoting them by
$\eta_\pm$~\cite{GK1,GKcohomology}. Svetlana Katok \cite{Katok91} used these operators to study cohomological equations for geodesic flows on negatively curved surfaces, including metrics of variable curvature. Using a related method, Flaminio \cite{Fla92} showed that, for the geodesic flow of a compact hyperbolic surface, real even invariant distributions with vanishing zeroth Fourier mode are determined by their values on holomorphic differentials of even degree. He also proved that even invariant distributions for the geodesic flow of a positively curved two-sphere are determined by their zeroth Fourier mode, a result later made quantitative by Bohr, Lefeuvre, and Paternain \cite{BLP24}. These operators have also been widely used to study cohomological equations and their applications to spectral and marked length spectrum rigidity. The related energy methods were developed by Pestov and
Sharafutdinov \cite{PestovSharafutdinov88,Sharafutdinov94}
and used by Croke--Sharafutdinov \cite{CS98} and
Dairbekov--Sharafutdinov \cite{DS03} in spectral rigidity
and integral geometry.
Paternain--Salo--Uhlmann \cite{PSU13,PSU14,PSU15} developed
the connections between Fourier estimates, invariant
distributions, and tensor tomography, while Guillarmou
\cite{Guillarmou17} brought microlocal methods into this
analysis. These ideas enter the local marked length spectrum
rigidity theorem of Guillarmou--Lefeuvre \cite{GL19} and
the rigidity theorem for Anosov surfaces of
Guillarmou--Lefeuvre--Paternain \cite{GLP25}. 
Paternain made extensive use of the Guillemin--Kazhdan operators with several different authors \cite{Paternain09, MP19}.
In the weighted marked length spectrum problem,
Gogolev and F.~Rodriguez Hertz \cite{GRH22} use the
cohomological equation and the Guillemin--Kazhdan theorem
to identify the weights.

On translation surfaces, the Guillemin--Kazhdan raising and lowering operators act on angular
Fourier coefficients as Forni's Cauchy--Riemann operators, up to
normalization~\cite[Section~2]{Forni21}. On flat cone surfaces with non-rational holonomy, Forni and Moll used
angular raising and lowering operators to prove that geodesic-flow
invariants in their horizontal Sobolev space are constant~\cite[Corollary~4.2]{FM25}. OpenAI used the angular raising and lowering operators to prove ergodicity
of the billiard flow in every triangle with an angle that is an irrational
multiple of $\pi$~\cite{OpenAI26}.

To the best of our knowledge, Lefeuvre \cite{Lefeuvre26}
was the first to apply the Guillemin--Kazhdan Fourier framework
to the singular indicator on closed nonpositively curved
surfaces. His argument and ours both use flow invariance to
relate successive even Fourier modes, the Guillemin--Kazhdan
identity to compare their derivatives, and square summability.
Lefeuvre seeks regularity of the invariant function itself.
He combines these tools with microlocal estimates for parabolic
dynamics and averaging to obtain additional anisotropic Sobolev
regularity for invariant $L^2$ functions, including the singular
indicator. Every angular Fourier mode belongs to $H^s$ for
all $s<5/8$. He also uses the Beurling transform and interpolation
to control the growth of Sobolev norms across successive modes
\cite[Section~4.2]{Lefeuvre26}. The horizontal operator
$L=-(X^2+H^2)$ also appears in his definition of the Sobolev norms.

Our argument, in turn, obtains the regularity needed for the
Guillemin--Kazhdan identity by applying heat to the Fourier
modes. We can then apply the identity at every positive heat
time. This regularization comes at a cost: heat need not preserve
flow invariance, so the relations between successive modes
acquire an error. Since the number of comparisons will grow,
we need to control the sum of these errors.

The commutator formulas express the error through products
of curvature with derivatives of the heated modes, and of
the curvature gradient with the heated modes themselves.
Both curvature and its gradient vanish on the common support
of the original modes. Near this support, we use the smallness
of these coefficients. Farther away, we use estimates showing
that the heated modes and their derivatives have little mass
for short heat times. Together, these estimates control the
accumulated error.

To compare the zeroth mode with modes of arbitrarily high degree,
we have to add an increasing number of errors. Shortening the
heat time reduces these errors, but also makes the derivative
bound larger. We therefore need an estimate that allows both
the accumulated error and the derivative at the final degree
to become small. 

We can then choose the heat time inversely proportional to the
degree. Parseval's identity makes the high-degree contribution
tend to zero along a subsequence, and the commutator estimate
bounds the remaining error by a quantity that can be made
arbitrarily small. Passing to the limit shows that the zeroth
mode has zero distributional gradient. It is therefore constant,
and its vanishing where curvature is negative makes this
constant zero. Since the zeroth mode is the fiber average of
the nonnegative singular indicator, the indicator vanishes
almost everywhere.

\medskip

We provide a very detailed proof of Theorem~\ref{thm:main} since this theorem has a history
of failed proof attempts. In Section~\ref{sec:outline}, we announce the extension of this
result to higher dimensions. Its proof will appear in a forthcoming
paper. In another forthcoming paper with Y.~Pesin, we will prove,
for smooth metrics, that the angular conditional measures of
invariant probability measures supported on the singular set
have zero Hausdorff dimension. That paper will also treat the
case of no focal points and related questions.

\subsection*{Acknowledgements}

We thank Yasha Pesin and Amie Wilkinson for their thorough reading of the manuscript and for their
detailed comments and suggestions, which helped us improve the exposition. We also thank Andrey
Gogolev, Raúl Ures and Alberto Verjovsky for their useful comments and support.

\subsection*{Disclosure of AI assistance}

The starting point was one author's proposal to regularize the
singular indicator by heat. Independently, another author was
working with heat and the Guillemin--Kazhdan operators on a
different problem. It came naturally to the team to explore this
combination of ideas in the present setting.

Discussions with the Astra model then focused on the estimates
needed to make this approach work. The $t^{-1/2}$ derivative bound
offered a way to control the derivative at a high Fourier degree
using square summability, provided one could also control the
accumulated commutator error.
The commutator formulas involve curvature and its gradient, both
of which vanish on the common support of the Fourier modes.
To use this vanishing, one needed to control how much of the heated
modes and their derivatives could spread away from that support.
In response to the authors' question about localizing the heat
evolution near this support, the model suggested Kannai's
heat--wave formula \eqref{proof:HWT} in Proposition~\ref{proof:weighted-prop}. This suggestion led to the use
of finite wave propagation to obtain the weighted estimates
in the proof.

AI also assisted with drafting, revising, and editing the manuscript.
The authors take full responsibility for the mathematical arguments,
the exposition---including the selection and ordering of concepts---and
the references.

\section{Outline of the proof}\label{sec:outline}

\subsection{Goal and strategy}\label{sec:strategy}

By Pesin's ergodicity theorem on the regular set
$\mathcal R=SM\setminus\mathrm{Sing}$ \cite{Pesin77}, it suffices to prove
$m(\mathrm{Sing})=0$. Put $f=\mathbf1_{\mathrm{Sing}}$ and let
\[
 f_0(x)=\frac1{2\pi}\int_{S_xM}f(x,v)\,d\theta
\]
be its average over each unit tangent circle. Our goal is to prove
$f_0=0$. Since $f\ge0$, a zero average on almost every fiber implies
$f=0$ almost everywhere. We will reach this goal by proving that the
distributional gradient of $f_0$ vanishes on $M$. Connectedness will then make
$f_0$ constant. This constant must be zero because $f_0=0$ on the nonempty
open set $\{K<0\}$: no singular vector can be based there, and this set
is nonempty by Gauss--Bonnet.

Passing to the orientable double cover, we may assume
that $M$ is oriented. Indeed, the pullback metric preserves
regularity and nonpositive curvature, and the cover
has negative Euler characteristic. Its geodesic flow
projects to that of $M$, carrying normalized Liouville
measure to normalized Liouville measure.
Thus ergodicity of the lifted flow implies ergodicity
of the original flow. We fix an orientation henceforth.

Let $X$ denote the vector field generating the geodesic flow,
let $V$ generate rotation in each oriented unit tangent circle
(locally, $V=\partial_\theta$), and set $H=[V,X]$. To measure derivatives of functions that are only in $L^2$, we use the
following weak norm. Let $H^1(SM)$ be the space of $L^2$ functions with
first weak derivatives in $L^2$, with its usual norm:
\[
 \|u\|_{H^1}
 =\bigl(\|u\|_2^2+\|Xu\|_2^2+\|Hu\|_2^2+\|Vu\|_2^2\bigr)^{1/2}.
\]
We denote the dual of $H^1(SM)$, using the $L^2$ pairing,
by $H^{-1}(SM)$. Its elements are distributions whose action
on test functions extends continuously to $H^1(SM)$.
The index $-1$ denotes negative Sobolev order. This space
contains the first distributional derivatives of $L^2$
functions: integration by parts transfers the derivative
to the test function, whose first derivatives are in $L^2$.
We equip this space with the dual norm
\begin{equation}\label{eq:negative-norm}
 \|u\|_{H^{-1}}
 =\sup_{\substack{\psi\in H^1(SM)\\\|\psi\|_{H^1}\le1}}
       |\langle u,\psi\rangle|.
\end{equation}
In particular, $L^2(SM)\subset H^{-1}(SM)$ and
$\|u\|_{H^{-1}}\le\|u\|_2$ for $u\in L^2(SM)$.
We regard a function on $M$, such as $f_0$, as a function on $SM$. Thus our goal can equivalently be
written as
\[
 \|f_0\|_{H^{-1}}=0.
\]
The estimates will prove that a first derivative of $f_0$ has
zero $H^{-1}$ norm. 

Here is how we obtain that derivative estimate. Flow invariance relates
neighboring angular Fourier modes. After applying heat, these relations
bound a derivative of the zeroth mode by a derivative of a high Fourier
mode and the accumulated error introduced by regularization. The heat
gradient estimate bounds the high-mode derivative by the original
Fourier coefficient. We take the heat time inversely proportional to
the Fourier degree: Parseval's identity makes this last term tend to
zero along a subsequence, while the curvature estimates make the
accumulated error arbitrarily small. We pass to the limit in $H^{-1}$,
where differentiation of an $L^2$ approximation is continuous.
Section~\ref{proof:completion-section} carries out these limits
and concludes that $f_0=f=0$ almost everywhere.

We use $\lesssim$ and $\gtrsim$ to suppress fixed multiplicative constants that are independent of
the Fourier degree and the heat time.

\subsection{The Fourier modes and their common support}

Let $X$ denote the geodesic vector field. Since the flow
preserves both $f$ and Liouville measure,
\[
 \int_{SM} f\,X\psi\,dm=0
 \qquad\text{for every smooth test function }\psi.
\]
This is the meaning of $Xf=0$ in the sense of distributions.
Invariance under reversal of direction gives $f(x,-v)=f(x,v)$.
Both identities are justified in
Section~\ref{subsec:indicator.function}.

To turn flow invariance into relations between neighboring angular
frequencies, decompose $f$ into Fourier modes,
\[
f=\sum_{k\in\mathbb Z}f_k,\qquad f_k\in\Omega_k,
\qquad \Omega_k=\{u\in L^2(SM):Vu=iku\},\quad V=\partial_\theta.
\]
In a local oriented orthonormal frame, the projection is
\[
f_k(x,\theta)=\frac1{2\pi}\int_0^{2\pi}
 f(x,\varphi)e^{ik(\theta-\varphi)}\,d\varphi.
\]
Reversal sends $\theta$ to $\theta+\pi$, so all odd modes vanish.
Parseval gives
\begin{equation}\label{eq:parseval}
\sum_{k\in\mathbb Z}\|f_k\|_2^2=\|f\|_2^2=m(\mathrm{Sing})<\infty.
\end{equation}

The zeroth Fourier mode is precisely the fiber average $f_0$ introduced
above. By the triangle inequality,
\[
 |f_k(x,\theta)|
 \le \frac1{2\pi}\int_0^{2\pi}|f(x,\varphi)|\,d\varphi
 =f_0(x),
\]
where the last equality uses $f\ge0$. Thus all modes $f_k$ are supported over the closed set
\[
F=\operatorname{supp}(f_0\,d\bar a)\subset\{K=0\}.
\]
 Since $\mathrm{Sing}$ is closed and $SM$ is compact, its projection
is closed and contains $F$. This projection need not coincide with
$F$, and neither set need be equal to the entire zero-curvature locus: a flat
geodesic may contribute only a set of zero angular measure, and a point
where $K=0$ need not lie on a complete flat geodesic.

If $F$ is empty, then $f_0=f=0$ almost everywhere and the proof is
complete. Otherwise the distance to $F$ used below is well defined.

The Fourier decomposition and the support properties of the modes
are proved in Section~\ref{subsec:indicator.function}.

\subsection{The Fourier recurrence and the energy identity}

The geodesic vector field changes angular degree by one. To separate its
two frequency shifts, let $H=[V,X]$ be the horizontal vector field that
moves the base point orthogonally to $v$, while parallel transporting $v$,
and define the Guillemin--Kazhdan operators
\[
\eta_+=\tfrac12(X-iH),\qquad
\eta_-=\tfrac12(X+iH).
\]
Since $f_0$ was defined as the fiber average of
$\mathbf1_{\mathrm{Sing}}$, its values lie in $[0,1]$.
In particular, $\overline{f_0}=f_0$. With derivatives
understood in the sense of distributions, we therefore have
$\eta_-f_0=\overline{\eta_+f_0}$.
Thus proving $\eta_+f_0=0$ also gives $\eta_-f_0=0$,
and hence $Xf_0=Hf_0=0$.
Since $f_0$ depends only on the base point and $X,H$
differentiate in two orthogonal base directions,
its distributional gradient on $M$ then vanishes.
It therefore suffices to prove
\[
 \|\eta_+f_0\|_{H^{-1}}=0.
\]
The distributional derivative $\eta_+f_0$ belongs to
$H^{-1}(SM)$ because $f_0\in L^2(SM)$. Indeed, the formal $L^2(SM,m)$ adjoint relation
$\eta_\pm^*=-\eta_\mp$ gives, by integration by parts,
\[
 |\langle\eta_\pm u,\psi\rangle|
 =|\langle u,\eta_\mp\psi\rangle|
 \lesssim\|u\|_2\|\psi\|_{H^1},
 \qquad u\in L^2(SM),\quad\psi\in H^1(SM).
\]
Taking the supremum in \eqref{eq:negative-norm}, we obtain
\begin{equation}\label{eq:weak-derivative}
 \|\eta_\pm u\|_{H^{-1}}\lesssim\|u\|_2.
\end{equation}
Consequently, convergence of functions in $L^2$ implies convergence of
these derivatives in $H^{-1}$. This is the reason for using the $H^{-1}$ norm defined in \eqref{eq:negative-norm}.

The operators $\eta_+$ and $\eta_-$ raise and lower angular degree by one,
respectively, and $X=\eta_++\eta_-$. Taking the Fourier mode of degree $2j+1$ in the distributional
equation $Xf=0$, and using the degree shifts of $\eta_\pm$, gives
\begin{equation}\label{eq:recurrence}
\eta_+f_{2j}+\eta_-f_{2j+2}=0.
\end{equation}
We call these relations between consecutive even Fourier modes
the \emph{Fourier recurrence}.

The sign of the curvature provides a comparison within each mode. In our setting, the Guillemin--Kazhdan energy identity
\cite[Lemma~B.2 and the proof of Lemma~B.3]{PSU15}
takes the following form: For
$u_k\in H^1(SM)\cap\Omega_k$, 
\begin{equation}\label{eq:energy}
\|\eta_+u_k\|_2^2-\|\eta_-u_k\|_2^2
=-\frac{k}{2}\int_{SM}K|u_k|^2\,dm.
\end{equation}
For $k\ge0$ and $K\le0$, it follows that
$\|\eta_-u_k\|_2\le\|\eta_+u_k\|_2$.
Formally combining this with \eqref{eq:recurrence} gives
\[
\|\eta_+f_0\|_2=\|\eta_-f_2\|_2
\le\|\eta_+f_2\|_2=\|\eta_-f_4\|_2\le\cdots.
\]
We call this sequence of comparisons the \emph{Fourier chain}.
We will read this chain as an upper bound on the derivative at degree
zero in terms of a derivative at a high degree. The latter can be small
because the Fourier coefficients are square summable. Two difficulties
remain: the derivatives displayed here need not be in $L^2$, and small
coefficients alone do not control their derivatives. We will apply heat to solve the lack of regularity.
But this will come at a cost: we will lose invariance,
and hence an error will appear in the Fourier chain.\par
Write $u_k(t)$ for the heated modes (see Section \ref{subsec:heating}). The recurrence now becomes
\[
\eta_+u_{2\ell}(t)+\eta_-u_{2\ell+2}(t)=r_\ell(t),
\]
where $r_\ell(t)$ is the error introduced by heat.
The Guillemin--Kazhdan inequality $\|\eta_-u_k\|_2\le\|\eta_+u_k\|_2$ still applies, so
\[
\|\eta_+u_{2\ell}(t)\|_2
\le \|\eta_-u_{2\ell+2}(t)\|_2+\|r_\ell(t)\|_2
\le \|\eta_+u_{2\ell+2}(t)\|_2+\|r_\ell(t)\|_2.
\]
Thus each step adds an error:
\[
\begin{aligned}
\|\eta_+u_0(t)\|_2
&\le \|\eta_+u_2(t)\|_2+\|r_0(t)\|_2\\
&\le \|\eta_+u_4(t)\|_2+\|r_0(t)\|_2+\|r_1(t)\|_2\\
&\le \cdots\\
&\le \underbrace{\|\eta_+u_{2j}(t)\|_2}_{\text{last derivative}}
+\underbrace{\sum_{\ell=0}^{j-1}\|r_\ell(t)\|_2}_{\text{accumulated error}}.
\end{aligned}
\]
We need to make both terms on the right small.
The heat derivative estimate \eqref{eq:smoothing} gives
\[
 \|\eta_+u_{2j}(t)\|_2
 \lesssim t^{-1/2}\|f_{2j}\|_2.
\]
The main task is therefore to bound the accumulated error
and choose a heat time that makes both terms small.
This will allow us to conclude that $\eta_+f_0=0$.

Our strategy is to show that the heated modes and their derivatives
remain concentrated in a small neighborhood of $F$.
Inside this neighborhood, we will control the error using the smallness
of curvature and its gradient, as explained in the next subsection.
Outside it, we will control the error using the rapid decay of the
heated modes and their derivatives.

For the Guillemin--Kazhdan operators and identities, see
Section~\ref{sec:fourier-proof}. The successive comparisons
between heated modes are carried out in
Section~\ref{proof:completion-section}.

\subsection{Quadratic vanishing of curvature}

To control the accumulated error, we first quantify how curvature vanishes near $F$.
The metric regularity gives $K\in C^{1+\mathrm{Lip}}$. Every $y\in F$
satisfies $K(y)=0$ and is a maximum of $K$, since $K\le0$. Thus
$\nabla K(y)=0$.

Given $x\in M$, choose a nearest point $y\in F$ and join it to $x$ by
a minimizing geodesic. Lipschitz continuity of $\nabla K$, compared by
parallel transport along this geodesic, bounds the gradient at distance
$s$ from $y$ by a constant times $s$. Integration then gives
\begin{equation}\label{eq:curvature}
|K(x)|\lesssim d_F(x)^2,\qquad
|\nabla K(x)|\lesssim d_F(x),\qquad
 d_F(x)=\operatorname{dist}(x,F).
\end{equation}
These curvature bounds are proved in
Section~\ref{sec:fourier-proof}; see \eqref{proof:curvature}.

\subsection{Heat regularization and the error induced in the Fourier chain}\label{subsec:heating}
A classical tool for improving the regularity of a function is to
let it evolve under the heat equation. Heat smooths out irregularities,
while the evolving function approaches the initial data in $L^2$
as the heat time tends to zero. We will apply this procedure separately
to each Fourier mode. At every positive time, the heated modes will
have the $H^1$ regularity needed to apply the Guillemin--Kazhdan
identity \eqref{eq:energy}, and their derivatives under $\eta_\pm$ will be bounded in terms of the
$L^2$ norms of the original modes. We will gain the regularity we need,
but lose invariance, introducing an error into the Fourier chain.

The fields $X$ and $H$ move the base point in two perpendicular
directions while parallel transporting the unit tangent vector.
These motions are called \emph{horizontal}, and differentiation
along $X$ and $H$ gives the horizontal derivatives.
The operator
\[
 L=-(X^2+H^2)
\]
is therefore called the \emph{horizontal Laplacian}.
We use its heat operator $P_t=e^{-tL}$ to regularize the Fourier
modes. Both $L$ and $P_t$ preserve Fourier degree.
The operator $L$ commutes with fiber rotations. Its restriction $L_k$ to
$\Omega_k$ is an elliptic operator on the corresponding line bundle over
$M$. Write
\[
P_t^k=e^{-tL_k},\qquad u_k(t)=P_t^kf_k.
\]
For $t>0$, these functions have the regularity needed for the Guillemin--Kazhdan energy identity 
\eqref{eq:energy}, and $u_k(t)\to f_k$ in $L^2$ as $t\downarrow0$. The heat regularization is introduced in Section~\ref{subsec:heat.regularization}. 

The spectral theorem gives, uniformly in $k$,
\begin{equation}\label{eq:smoothing}
\|\eta_\pm P_t^kf_k\|_2\lesssim t^{-1/2}\|f_k\|_2.
\end{equation}
This derivative estimate is
established in Section~\ref{subsec:spectral}.

Applying heat to the Fourier modes can introduce an error in the Fourier chain. This error is
\begin{equation}\label{proof:error-definition}
r_j(t)=\eta_+u_{2j}(t)+\eta_-u_{2j+2}(t).
\end{equation}
Using \eqref{eq:recurrence}, this becomes
\[
r_j(t)=[\eta_+,P_t]f_{2j}+[\eta_-,P_t]f_{2j+2},
\]
where, on a degree-$k$ input,
\[
[\eta_\pm,P_t]f_k
=\eta_\pm P_t^kf_k-P_t^{k\pm1}\eta_\pm f_k.
\]
The second term is initially interpreted distributionally. The preceding identity expresses each recurrence error $r_j(t)$
as the sum of two commutators. Controlling these commutators
will therefore control the error introduced at each step of
the Fourier chain. We will prove that, for an initial condition
supported over $F$, each commutator belongs to $L^2$ and satisfies
\begin{equation}\label{eq:commutator}
\|[\eta_\pm,P_t]f_k\|_2
\lesssim(1+|k|)t^{3/2}\|f_k\|_2,
\qquad \operatorname{supp}f_k\subset\pi^{-1}(F).
\end{equation}
Establishing this estimate is the main work still required.
It uses the vanishing of the curvature and its gradient on $F$, together with
estimates showing that heat starting from $F$ remains
concentrated near $F$ for short times. The next two subsections
explain how these facts enter the argument; the detailed proof
is completed in Proposition~\ref{proof:commutator-prop}
in Section~\ref{proof:weighted-section}.

The explicit factor $1+|k|$ records how the error bound grows
with the Fourier degree. This matters because we must add
the errors over an increasing number of steps.

For a chain ending at degree $2j$, we need both
\[
t^{-1/2}\|f_{2j}\|_2
\qquad\text{and}\qquad
\sum_{\ell=0}^{j-1}\|r_\ell(t)\|_2
\]
to be small. Taking a shorter heat time helps control the
second quantity, since each error bound contains $t^{3/2}$,
but increases the factor $t^{-1/2}$ in the first quantity.
We will therefore choose the heat time depending on $j$:
it must be short enough to keep the accumulated error small,
but not so short that the growing factor $t^{-1/2}$ prevents
the first quantity from tending to zero along a subsequence.
To make this choice, we first need to establish
\eqref{eq:commutator}, which is the purpose of the next two
subsections.

\subsection{Differentiating the error}

To estimate $r_\ell(t)$ in \eqref{proof:error-definition}, we
differentiate it with respect to the heat time. We do this because
$L=-(X^2+H^2)$ has an explicit differential expression, whereas $P_t$
is obtained by solving the heat equation. Using
$\partial_tu_k=-L_ku_k$, we obtain
\[
 \partial_t r_\ell+L_{2\ell+1}r_\ell
 =[L,\eta_+]u_{2\ell}+[L,\eta_-]u_{2\ell+2}.
\]
We derive this equation in Section~\ref{subsec:evolution.error}. The original recurrence \eqref{eq:recurrence} gives the zero initial
condition $r_\ell(0)=0$ in distributions. Thus the error is obtained by
integrating these commutator terms, with heat acting on each term for
the remaining time.

This is useful because $[L,\eta_\pm]$ can be computed directly, and
its coefficients involve only the curvature and its first derivatives.
Writing $\nabla_k$ for the horizontal covariant derivative in degree
$k$, the computation gives, for sufficiently regular $u\in\Omega_k$,
\[
 \|[L,\eta_\pm]u\|_2
 \lesssim(1+|k|)
 \bigl(\|K\nabla_k u\|_2+\|(\nabla K)u\|_2\bigr).
\]
If we use only boundedness of $K$ and $\nabla K$, the heat gradient
estimate gives
\[
 \|[L,\eta_\pm]P_s^kf_k\|_2
 \lesssim(1+|k|)(1+s^{-1/2})\|f_k\|_2.
\]
Integrating this bound, using that the heat operators have $L^2$
operator norm bounded by one, gives, for $0<t\le1$,
\[
 \|r_\ell(t)\|_2
 \lesssim(\ell+1)\sqrt t
 \bigl(\|f_{2\ell}\|_2+\|f_{2\ell+2}\|_2\bigr).
\]
Cauchy--Schwarz and Parseval would therefore give only
\[
 \sum_{\ell=0}^{j-1}\|r_\ell(t)\|_2
 \lesssim\sqrt t\,j^{3/2}\|f\|_2.
\]
Choosing smaller heat times reduces this error bound, but it also
weakens the derivative estimate \eqref{eq:smoothing}. Indeed, the
Fourier chain requires control of the derivative at degree $2j$,
whose bound is
\[
 \|\eta_+u_{2j}(t_j)\|_2
 \lesssim t_j^{-1/2}\|f_{2j}\|_2.
\]
We need to make both this bound and the accumulated error small.
To guarantee a small accumulated error through the preceding bound,
we would have to choose $t_j\lesssim \tau/j^3$, with $\tau>0$ small.
At $t_j=\tau/j^3$, the error bound is of order $\sqrt \tau\,\|f\|_2$,
uniformly in $j$, but the derivative bound becomes
\[
 \tau^{-1/2}j^{3/2}\|f_{2j}\|_2.
\]
Parseval does not ensure that this expression tends to zero even
along a subsequence: square summability alone allows
$\|f_{2j}\|_2=1/j$, for which the expression grows like $\sqrt j$.
Taking still smaller times only worsens this derivative bound.
We therefore need a sharper error estimate, so that controlling
the accumulated error does not require such short heat times.

The missing information is that the curvature vanishes quadratically on
the common support $F$. Using \eqref{eq:curvature} instead of mere
boundedness gives
\[
 \|[L,\eta_\pm]P_s^kf_k\|_2
 \lesssim(1+|k|)
 \bigl(\|d_F^2\nabla_kP_s^kf_k\|_2
       +\|d_FP_s^kf_k\|_2\bigr).
\]
We therefore need to control exactly these two distance-weighted
norms. They will be small if heat and its horizontal derivative
remain sufficiently concentrated near $F$, where the weights vanish.

The evolution equation and the bounds for its commutator terms
are derived in Section~\ref{proof:sec:commutator}.

\subsection{Concentration of heat near the common support}

The curvature is small near $F$, so we would like most of the mass
of the heated mode to remain near $F$ as well. At time zero, this
is certainly true: $f_k$ is supported in $F$. Heat spreads the
function beyond this set, but we expect that, for a short time,
very little mass can move far away. We need to understand
how far it spreads and how small the part farther away is.

The eigenfunction expansion does not make this easy to see.
Kannai's heat--wave formula \cite{Kannai} offers a different way
to look at the same heated function:
\begin{equation}\label{eq:kannai}
 P_t^kf_k=\frac1{\sqrt{\pi t}}\int_0^\infty
 e^{-r^2/(4t)}\cos(r\sqrt{L_k})f_k\,dr.
\end{equation}
It writes heat as an average of waves. This is useful because
we know how far each wave can have travelled. Starting from $F$,
the wave at time $r$ is supported in
\[
 F_r=\{x\in M:d_F(x)\le r\}.
\]
This is the finite propagation property: the wave cannot reach
a point at distance greater than $r$ from $F$ by time $r$.

Now consider a point $x$ far from $F$. By finite propagation,
the wave $\cos(r\sqrt{L_k})f_k$ vanishes near $x$ whenever
$r<d_F(x)$. Thus only the part of the integral in
\eqref{eq:kannai} over $r\ge d_F(x)$ can contribute at $x$. But the Gaussian factor
$e^{-r^2/(4t)}$ becomes very small when $r$ is much larger
than $\sqrt t$. Thus, to reach a distance much larger than
$\sqrt t$, we must use waves that receive very little weight
in Kannai's formula. Since the wave operators have norm bounded
by one, their contributions are small as well.

This gives us a way to show that heat stays concentrated within
a distance of order $\sqrt t$ from $F$, while allowing a small amount to spread farther away. We will use it below
to estimate both the heated mode and its derivative.

For the first weighted norm, the support bound gives
$d_F\le r$ wherever $\cos(r\sqrt{L_k})f_k$ is nonzero, while the
$L^2$ operator norm of this wave operator is bounded by one. Inserting these two facts into \eqref{eq:kannai} gives
the first weighted estimate, which controls the heated mode:
\[
 \|d_FP_t^kf_k\|_2\lesssim\sqrt t\,\|f_k\|_2.
\]
To control the derivative, use the wave
\[
 w(r)=\frac{\sin(r\sqrt{L_k})}{\sqrt{L_k}}f_k,
 \qquad w(0)=0,\quad w_r(0)=f_k.
\]
Here $w_r(r)=\cos(r\sqrt{L_k})f_k$. Both $w$ and its horizontal
derivative are supported over $F_r$, and conservation of wave energy
gives $\|\nabla_k w(r)\|_2\le\|f_k\|_2$. Integrating
\eqref{eq:kannai} by parts therefore gives
\[
 P_t^kf_k=\frac1{2t\sqrt{\pi t}}\int_0^\infty
 r e^{-r^2/(4t)}w(r)\,dr.
\]
Taking a horizontal derivative and using $d_F^2\le r^2$ on its
support yields the second weighted estimate, which controls
its horizontal derivative:
\[
 \|d_F^2\nabla_kP_t^kf_k\|_2
 \lesssim t^{-3/2}\int_0^\infty r^3e^{-r^2/(4t)}\,dr\,\|f_k\|_2
 \lesssim\sqrt t\,\|f_k\|_2.
\]
We can now bound the curvature terms over the whole surface.
Since $|K|\lesssim d_F^2$ and $|\nabla K|\lesssim d_F$,
the two weighted estimates give
\begin{align*}
 \|K\nabla_kP_s^kf_k\|_2
 &\lesssim\|d_F^2\nabla_kP_s^kf_k\|_2
 \lesssim\sqrt s\,\|f_k\|_2,\\
 \|(\nabla K)P_s^kf_k\|_2
 &\lesssim\|d_FP_s^kf_k\|_2
 \lesssim\sqrt s\,\|f_k\|_2.
\end{align*}
These inequalities combine the two facts we needed:
curvature and its gradient are small near $F$, while the heated function and its horizontal derivative
have little mass far from $F$ for short heat times.  The commutator identities on Fourier degree $k$ are
\[
 \begin{aligned}\relax
 [L,\eta_+]u
   &=-(2k+1)K\eta_+u-k(\eta_+K)u,\\
 [L,\eta_-]u
   &=(2k-1)K\eta_-u+k(\eta_-K)u.
 \end{aligned}
\]
Thus each commutator consists of curvature multiplying
a derivative of $u$, and a derivative of curvature
multiplying $u$. Applying these identities to $u=P_s^kf_k$
and using the preceding weighted estimates gives
\[
 \|[L,\eta_\pm]P_s^kf_k\|_2
 \lesssim(1+|k|)\sqrt s\,\|f_k\|_2.
\]
The implicit constant is independent of $k$ and $s$.

It remains to integrate the commutator contributions over
$0<s<t$. For each fixed $k$, the error $[\eta_\pm,P_t]f_k$
satisfies
\[
 (\partial_t+L_{k\pm1})[\eta_\pm,P_t]f_k
 =[L,\eta_\pm]P_t^kf_k.
\]
The right-hand side is the source term. Both terms defining
$[\eta_\pm,P_t]f_k$ converge to $\eta_\pm f_k$ in distributions
as $t\downarrow0$, so the error has zero initial value.
Duhamel's formula \eqref{proof:Duhamel} therefore expresses
the error as the integral of these source terms, each evolved
by heat for the remaining time $t-s$:
\[
 [\eta_\pm,P_t]f_k
 =\int_0^t P_{t-s}^{k\pm1}
       [L,\eta_\pm]P_s^kf_k\,ds.
\]

Since the heat operator has $L^2$ norm bounded by one,
this remaining evolution cannot increase the norm of
each contribution. Therefore
\begin{align*}
 \|[\eta_\pm,P_t]f_k\|_2
 &\le\int_0^t\|[L,\eta_\pm]P_s^kf_k\|_2\,ds\\
 &\lesssim(1+|k|)\|f_k\|_2\int_0^t\sqrt s\,ds\\
 &\lesssim(1+|k|)t^{3/2}\|f_k\|_2.
\end{align*}
This proves \eqref{eq:commutator}. The earlier estimate
used only boundedness of curvature and its gradient.
Using their vanishing on $F$ and the concentration
of heat improves the bound before integration from
$s^{-1/2}$ to $s^{1/2}$ for small $s$, and hence improves
the resulting error bound from $t^{1/2}$ to $t^{3/2}$.

Finite wave propagation and the resulting weighted heat and
commutator estimates are proved in
Section~\ref{proof:weighted-section}.
\subsection{An upper bound for the derivative at degree zero}

The Guillemin--Kazhdan energy identity \eqref{eq:energy} and the definition of $r_\ell(t)$ \eqref{proof:error-definition} give
\[
\|\eta_+u_{2\ell}(t)\|_2
\le\|\eta_-u_{2\ell+2}(t)\|_2+\|r_\ell(t)\|_2
\le\|\eta_+u_{2\ell+2}(t)\|_2+\|r_\ell(t)\|_2.
\]
Iterating at one fixed heat time $t>0$ yields
\begin{equation*}
\|\eta_+u_0(t)\|_2
\le\|\eta_+u_{2j}(t)\|_2+\sum_{\ell=0}^{j-1}\|r_\ell(t)\|_2.
\end{equation*}
By \eqref{eq:commutator},
\[
\|r_\ell(t)\|_2\lesssim t^{3/2}(\ell+1)
 \bigl(\|f_{2\ell}\|_2+\|f_{2\ell+2}\|_2\bigr).
\]
Cauchy--Schwarz and Parseval bound the sum of these errors by
$\lesssim(tj)^{3/2}\|f\|_2$. Together with \eqref{eq:smoothing}, this gives
\begin{equation}\label{eq:outline-upper-bound}
\|\eta_+u_0(t)\|_2
\lesssim t^{-1/2}\|f_{2j}\|_2+(tj)^{3/2}\|f\|_2.
\end{equation}
This upper bound is proved in Section~\ref{subsec:upper.bound}.

\subsection{The choice \texorpdfstring{$t=\tau/j$}{t=tau/j} and the two limits}

Fix an arbitrary small $\tau>0$ and take $t=\tau/j$.
Substituting into the accumulated-error bound gives
\[
 \sum_{\ell=0}^{j-1}\|r_\ell(\tau/j)\|_2
 \lesssim \bigl((\tau/j)j\bigr)^{3/2}\|f\|_2
 =\tau^{3/2}\|f\|_2.
\]
The right-hand side no longer depends on $j$. Thus choosing
$\tau$ small makes the accumulated error small, regardless of
how many steps we take in the Fourier chain.
For the derivative term, the same substitution gives
$t^{-1/2}=\sqrt{j/\tau}$. Consequently,
\eqref{eq:outline-upper-bound} becomes
\[
\|\eta_+u_0(\tau/j)\|_2
\lesssim \tau^{-1/2}\sqrt j\,\|f_{2j}\|_2+\tau^{3/2}\|f\|_2.
\]
This explains the choice of heat time. At each $j$ the recurrence is
iterated at the single time $\tau/j$; only after obtaining the estimate do
we let $j$ tend to infinity.

Square summability in \eqref{eq:parseval} implies
\[
\liminf_{j\to\infty}\sqrt j\,\|f_{2j}\|_2=0.
\]
Indeed, if
$\liminf_{j\to\infty}\sqrt j\,\|f_{2j}\|_2>0$,
then, for all sufficiently large $j$,
\[
 \|f_{2j}\|_2^2\gtrsim\frac1j.
\]
Summing over $j$ would contradict Parseval's identity,
because the harmonic series diverges whereas
$\sum_j\|f_{2j}\|_2^2<\infty$.
We can therefore choose a sequence $j_n\to\infty$ such that
\[
 \sqrt{j_n}\,\|f_{2j_n}\|_2\longrightarrow0.
\]

Keep $\tau>0$ fixed and apply \eqref{eq:outline-upper-bound}
with $t=\tau/j_n$. The high-degree term
$t^{-1/2}\|f_{2j_n}\|_2$ becomes
$\tau^{-1/2}\sqrt{j_n}\,\|f_{2j_n}\|_2$, which tends to zero.
The accumulated-error term remains bounded by
$\lesssim\tau^{3/2}\|f\|_2$, independently of $n$. Consequently,
\[
 \limsup_{n\to\infty}
 \|\eta_+u_0(\tau/j_n)\|_2
 \lesssim \tau^{3/2}\|f\|_2.
\]

We next transfer this estimate from the heated functions
to $f_0$. Since $\tau/j_n\to0$, heat convergence gives
$u_0(\tau/j_n)\to f_0$ in $L^2$. Applying
\eqref{eq:weak-derivative} to their difference yields
\[
 \|\eta_+u_0(\tau/j_n)-\eta_+f_0\|_{H^{-1}}
 \lesssim\|u_0(\tau/j_n)-f_0\|_2
 \longrightarrow0.
\]
Thus the derivatives converge in $H^{-1}$.
Since the $H^{-1}$ norm is bounded by the $L^2$ norm,
the preceding estimate implies
\[
 \|\eta_+f_0\|_{H^{-1}}
 \lesssim \tau^{3/2}\|f\|_2.
\]
The left-hand side does not depend on $\tau$. We may now let $\tau\downarrow0$
and conclude that $\eta_+f_0=0$ in distributions.

Since $f_0$ is real, taking real and imaginary parts of
$\eta_+f_0=0$ gives $Xf_0=Hf_0=0$. These are the
derivatives of $f_0$ in the two orthogonal base directions,
so its distributional gradient on $M$ vanishes.
Since $M$ is connected, $f_0$ is constant almost everywhere.
It vanishes on the nonempty open set $\{K<0\}$, so this
constant is zero. Finally, the definition of $f_0$ as the
average of $f=\mathbf1_{\mathrm{Sing}}$ gives
\[
 \|f\|_2^2=m(\mathrm{Sing})
 =\int_M f_0\,d\bar a=0.
\]
The regular set therefore has full Liouville measure,
and Pesin's ergodicity theorem on $\mathcal R$
\cite{Pesin77} gives ergodicity on all of $SM$.
The two limits and the deduction of ergodicity are justified in
Sections~\ref{subsec:two.limits} and \ref{the.end}.

\subsection{Higher dimensions}

The corresponding result for closed smooth rank-one manifolds of
nonpositive sectional curvature requires a substantially more involved
proof. Its geometric component is, in a significant sense, orthogonal to
the surface argument, rather than a direct extension of the Fourier and
heat estimates above. The full proof will be presented in a subsequent
paper; we give only a brief indication of the strategy here.

In dimension $n$, let $\mathcal S_r$ denote the stratum of vectors of
rank $r$. The conclusion is $m(\mathcal S_r)=0$ for $r=2,\ldots,n$,
and hence Liouville ergodicity.

At maximal rank, the $n-1$ parallel normal Jacobi fields force every plane
containing the geodesic direction to be flat, allowing a direct analogue
of the curvature and heat estimates above. Intermediate ranks require
further Guillemin--Kazhdan arguments together with an analysis of
accessibility classes and geometric continuation. If a rank stratum has
positive measure, one considers its essential support
\[
C=\operatorname{supp}(\mathbf1_{\mathcal S_r}m)
\subset\{v:\operatorname{rank}(v)\ge r\}.
\]
The essential support is saturated by complete stable and unstable
leaves, so it is a union of \emph{accessibility classes}: sets of vectors
that can be joined by concatenating paths in stable and unstable leaves.
The guiding philosophy is that $su$-saturated sets of positive measure
should have nonempty interior. Establishing this in the present setting
requires the additional analytic estimates and geometric continuation,
which show that $C$ has nonempty interior in $SM$.

This use of essential supports and stable/unstable saturation follows
the work of the second author of this paper, J.~Rodriguez Hertz
\cite[Section~6, especially Lemma~6.1]{JRH12}, where the strategy was
used for partially hyperbolic invariant sets with one-dimensional center. Its
implementation for geodesic rank is the principal geometric step beyond
the surface argument. Since $C$ contains no rank-one vectors, its
nonempty interior contradicts their density.

\section{Fourier modes and the connection Laplacian}\label{sec:fourier-proof}

We now give the detailed argument for the strategy in
Section~\ref{sec:strategy}: prove $\|\eta_+f_0\|_{H^{-1}}=0$,
then use the vanishing of $f_0$ on $\{K<0\}$ to conclude $f_0=f=0$.
All norms use the normalized measures and the Sobolev spaces introduced
above. For a function of one fixed Fourier
degree, its $L^2(SM)$ norm equals the $L^2(M,d\bar a)$ norm of its local
coefficient, since the angular factor has modulus one. Inner products
are linear in the first argument.

\subsection{The invariant function and its Fourier components}\label{subsec:indicator.function}

Set $f=\mathbf1_{\mathrm{Sing}}$. We will prove  that $f=0$
almost everywhere. The singular set is closed and  flow invariant. Consequently,
\begin{equation}\label{proof:invariant}
 Xf=0\quad\text{in distributions}.
\end{equation}
Indeed, Liouville invariance and $f\circ g^s=f$ give
$$\int_{SM} f\,\psi\circ g^s\,dm=\int_{SM}f\psi\,dm$$ for every
differentiable test function $\psi$. Differentiating at $s=0$ gives
$\int_{SM}fX\psi\,dm=0$, which is the assertion in
\eqref{proof:invariant}. 

For each $x\in M$, the fiber $S_xM$ consists of the unit tangent
vectors at $x$. Choose a local positively oriented orthonormal
frame $(e_1,e_2)$ and write
\[
v=\cos\theta\,e_1(x)+\sin\theta\,e_2(x),
\qquad \theta\in\mathbb R/2\pi\mathbb Z.
\]
The vector field $V$ differentiates with respect to this angle,
keeping $x$ fixed:
\[
Vu(x,\theta)=\partial_\theta u(x,\theta).
\]
Changing the frame changes the origin of the angle but not
$\partial_\theta$, so $V$ is globally defined on $SM$.

For $k\in\mathbb Z$, let
\[
\Omega_k
=\{u\in L^2(SM):Vu=iku\text{ in distributions}\}.
\]
The distributional interpretation means that
\[
-\int_{SM}u\,\overline{V\varphi}\,dm
=ik\int_{SM}u\,\overline{\varphi}\,dm
\qquad\text{for every }\varphi\in C^\infty(SM),
\]
where $m$ is Liouville measure; thus no differentiability of $u$
is assumed. Equivalently, $u\in\Omega_k$ precisely when its local
expression has the form
\[
u(x,\theta)=\widehat u_k(x)e^{ik\theta}
\]
almost everywhere, with $\widehat u_k$ independent of $\theta$.
Hence $\Omega_k$ is the space of functions containing only
the $k$th angular Fourier mode.
For $f=\mathbf1_{\mathrm{Sing}}$, the Fourier projection of $f$ in the circle
variable gives
\[
 f_k(x,\theta)=\widehat f_k(x)e^{ik\theta},\qquad
 \widehat f_k(x)=\frac1{2\pi}\int_0^{2\pi}
                  f(x,\varphi)e^{-ik\varphi}\,d\varphi.
\]
The spaces $\Omega_k$ are
mutually orthogonal, and Fourier expansion on each circle, followed by
integration over the base, gives
\begin{equation*}
 f=\sum_{k\in\mathbb Z}f_k\quad\text{in }L^2(SM),\qquad
 \sum_{k\in\mathbb Z}\|f_k\|_2^2=\|f\|_2^2=m(\mathrm{Sing}).
\end{equation*}
Since $f(x,v)=f(x,-v)$, we can replace $\theta$ by $\theta+\pi$. Its invariance therefore
implies $f_k=(-1)^kf_k$, so $f_k=0$ for odd $k$.

The zeroth coefficient and the pointwise coefficient bound are
\[
 f_0(x)=\frac1{2\pi}\int_0^{2\pi}f(x,\theta)\,d\theta\ge0,
 \qquad |f_k(x,\theta)|\le f_0(x)
 \quad\text{for almost every }x.
\]
If $f_0=0$ almost everywhere, the coefficient bound gives $f_k=0$
for every $k$, hence $f=0$, and there is nothing to prove. Otherwise,
the closed set in the following definition is nonempty:
\begin{equation*}
 F=\operatorname{supp}(f_0\,d\bar a),\qquad
 d_F(x)=\operatorname{dist}(x,F).
\end{equation*}
The function $f_0$ vanishes almost everywhere on $M\setminus F$.
The coefficient bound shows that every $f_k$ vanishes over $M\setminus F$.

No vector based at a point where $K<0$ is singular. Hence $f_0=0$ on
this open set and $F\subset\{K=0\}$. By Gauss--Bonnet and the genus
assumption, $\{K<0\}$ is nonempty. Thus, once we prove that $f_0$
is constant, its vanishing on this open set will imply $f_0=0$.

The regularity of the metric gives $K\in C^{1+\mathrm{Lip}}$. Every
point of $F$ is a maximum of $K$, so $K=0$ and $\nabla K=0$ on $F$.
Choose a nearest point $y\in F$ to $x$ and a minimizing geodesic
from $y$ to $x$. Along this geodesic, Lipschitz continuity of $\nabla K$
gives $|\nabla K|\lesssim s$ at distance $s$ from $y$; integrating along the geodesic gives $|K|\lesssim s^2$. Therefore
\begin{equation}\label{proof:curvature}
 |K(x)|\lesssim d_F(x)^2,\qquad
 |\nabla K(x)|\lesssim d_F(x).
\end{equation}
Compactness makes the constants global. 

\subsection{The Guillemin--Kazhdan operators}

Let $J$ be rotation through $\pi/2$ in the oriented tangent planes, and
let $H=[V,X]$, where the Lie bracket is defined by
\[
Hu=V(Xu)-X(Vu)
\]
for smooth functions $u$. At $(x,v)$, the field $H$ moves the base point in the
direction $Jv$ while parallel transporting $v$. With the convention $H=[V,X]$, computation using the Levi--Civita connection
gives the commutation relations
\cite{GK1,GKcohomology}
\begin{equation}\label{proof:brackets}
 [V,X]=H,\qquad [V,H]=-X,\qquad [X,H]=KV.
\end{equation}
Here the first identity is the definition of $H$.
The geodesic and rotation flows preserve Liouville measure, so $X$ and
$V$ have zero divergence. Their bracket $H$ does as well. Since $X,H,V$ preserve Liouville measure, integration by parts gives
\[
\langle Xu,v\rangle=-\langle u,Xv\rangle,
\qquad u,v\in C^1(SM),
\]
and the same identity holds with $X$ replaced by $H$ or $V$.

To see how $X$ acts on Fourier modes, fix $x\in M$ and choose
a local oriented orthonormal frame $(e_1,e_2)$ whose covariant
derivatives vanish at $x$. At this point, parallel transport
keeps the angle $\theta$ fixed to first order. Thus, writing
$\partial_j$ for differentiation along $e_j$ with $\theta$
held fixed, we have on the fiber $S_xM$
\begin{equation}\label{eq:horizontal-frame}
 X=\cos\theta\,\partial_1+\sin\theta\,\partial_2,
 \qquad
 H=-\sin\theta\,\partial_1+\cos\theta\,\partial_2.
\end{equation}
For a smooth function $u\in\Omega_k$, locally written as
$u(x,\theta)=a(x)e^{ik\theta}$, this gives
\begin{align*}
Xu
&=\bigl(\cos\theta\,\partial_1a
       +\sin\theta\,\partial_2a\bigr)e^{ik\theta}\\
&=\frac12(\partial_1a-i\partial_2a)e^{i(k+1)\theta}
 +\frac12(\partial_1a+i\partial_2a)e^{i(k-1)\theta}.
\end{align*}
The first term has Fourier degree $k+1$ and the second has
degree $k-1$. Using the expression for $H$, these terms are
respectively $\tfrac12(X-iH)u$ and $\tfrac12(X+iH)u$.
This leads to the Guillemin--Kazhdan operators
\cite{GK1,GKcohomology}:
\begin{equation}\label{proof:eta}
 \eta_+=\tfrac12(X-iH),\qquad
 \eta_-=\tfrac12(X+iH).
\end{equation}
They are globally defined because $X$ and $H$ are global
vector fields. Since the computation applies at every base
point, $\eta_+$ raises Fourier degree by one and $\eta_-$
lowers it by one. These degree shifts also hold
distributionally for nonsmooth inputs. Since
$X=\eta_++\eta_-$, the degree-$(2j+1)$ part of
\eqref{proof:invariant} is
\begin{equation}\label{proof:recurrence}
 \eta_+f_{2j}+\eta_-f_{2j+2}=0
 \quad\text{in distributions},\qquad j\ge0.
\end{equation}
These Fourier recurrences appear in Guillemin--Kazhdan \cite{GKcohomology},
Paternain--Salo--Uhlmann \cite[Sections~2 and~5]{PSU14}, and, for
invariant functions in nonpositive curvature, Lefeuvre
\cite[Section~4.2]{Lefeuvre26}.

Since $X$ and $H$ preserve Liouville measure, their formal
$L^2(SM,m)$ adjoints satisfy $X^*=-X$ and $H^*=-H$.
Using the definitions \eqref{proof:eta}, and conjugating
scalar coefficients when taking adjoints, we obtain
\begin{align*}
\eta_+^*
&=\tfrac12(X^*+iH^*)
=-\tfrac12(X+iH)
=-\eta_-,\\
\eta_-^*
&=\tfrac12(X^*-iH^*)
=-\tfrac12(X-iH)
=-\eta_+.
\end{align*}
Expanding the products in \eqref{proof:eta} gives
\begin{align*}
4[\eta_+,\eta_-]
&=(X-iH)(X+iH)-(X+iH)(X-iH)\\
&=(X^2+iXH-iHX+H^2)
 -(X^2-iXH+iHX+H^2)\\
&=2i(XH-HX)
=2i[X,H].
\end{align*}
Using $[X,H]=KV$ from \eqref{proof:brackets}, we conclude that
\begin{equation*}
[\eta_+,\eta_-]=\tfrac i2KV.
\end{equation*}
For every $u\in H^1(SM)\cap\Omega_k$, the
Guillemin--Kazhdan energy identity \cite{GK1} states that
\begin{equation}\label{proof:GK}
\|\eta_+u\|_2^2-\|\eta_-u\|_2^2
=-\frac{k}{2}\int_{SM}K|u|^2\,dm.
\end{equation}
In particular, because $K\le0$,

\begin{equation}\label{proof:monotone}
 \|\eta_+u\|_2\ge\|\eta_-u\|_2
 \qquad(u\in H^1(SM)\cap\Omega_k,\ k\ge0).
\end{equation}

\subsection{The connection Laplacian on one Fourier degree}\label{subsec:heat.regularization}
Heat evolution is a classical way to smooth a rough function:
diffusion damps rapid spatial oscillations, making the function
more regular. At the same time, the solution converges to the
initial function in $L^2$ as the heat time tends to zero.
We apply this idea to each Fourier mode, obtaining the $H^1$
regularity needed for the Guillemin--Kazhdan energy identity
while keeping the regularized modes arbitrarily close to the
original ones in $L^2$.

Consider the horizontal Laplacian
\begin{equation}\label{proof:L}
 L=-(X^2+H^2)=-2(\eta_+\eta_-+\eta_-\eta_+).
\end{equation}

\begin{claim}
    $L$ preserves Fourier degree. We denote its restriction to $\Omega_k$ by $L_k$. 
\end{claim} 

\begin{proof}
A function has Fourier degree $k$ precisely when $Vu=iku$.
We will show that applying $L$ preserves this relation.

First, the commutation relations \eqref{proof:brackets} give
\begin{align*}
[V,L]
&=-[V,X^2]-[V,H^2]\\
&=-\bigl([V,X]X+X[V,X]\bigr)
  -\bigl([V,H]H+H[V,H]\bigr)\\
&=-(HX+XH)+(XH+HX)=0.
\end{align*}
Consequently, for $u\in H^2(SM)\cap\Omega_k$,
\[
V(Lu)=L(Vu)=L(iku)=ik\,Lu,
\]
with the derivatives understood in distributions.
Since $Lu\in L^2(SM)$, this says exactly that
$Lu\in\Omega_k$.
\end{proof}

On $\Omega_k$, let $\nabla_k$ denote the covariant derivative
induced by the Levi--Civita connection: differentiate $u(x,v)$
as $x$ moves along a curve and $v$ is parallel transported
along that curve.

\begin{claim}
 For every $k\in\mathbb Z$, the operator
$L_k$ is elliptic on the base $M$.
\end{claim}
\begin{proof}
By \eqref{eq:horizontal-frame} and the definition of $\nabla_k$,
$Xu$ and $Hu$ represent the covariant derivatives of $u$ in the
base directions $v$ and $Jv$, respectively. Since these directions
form an orthonormal basis of $T_xM$, we have
\[
 |\nabla_k u|^2(x)=|Xu(x,v)|^2+|Hu(x,v)|^2.
\]
Integration by parts therefore gives for $u\in H^2(SM)\cap\Omega_k$,
\begin{equation}\label{proof:energy}
 \langle L_ku,u\rangle
 =\|Xu\|_2^2+\|Hu\|_2^2
 =\|\nabla_k u\|_2^2
 =2\|\eta_+u\|_2^2+2\|\eta_-u\|_2^2.
\end{equation}
Thus $L_k=\nabla_k^*\nabla_k$ is a connection Laplacian
on $M$, hence elliptic.
\end{proof}
On $\Omega_k$, the operator $L_k$ with domain
$H^2(SM)\cap\Omega_k$ is nonnegative and self-adjoint.
For $f_k\in\Omega_k$, consider the heat equation
\begin{equation}\label{eq:heat}
\partial_t u_k(t)=-L_k u_k(t),\qquad u_k(0)=f_k.
\end{equation}
We denote its solution operator by $P_t^k=e^{-tL_k}$, so that
\[
u_k(t)=P_t^k f_k=e^{-tL_k}f_k.
\]
Thus $e^{-tL_k}$ means evolution under the heat equation
for time $t$. The solution converges to its initial condition in $L^2(SM)$:
\[
\|u_k(t)-f_k\|_2\longrightarrow0
\qquad\text{as }t\downarrow0.
\]

\begin{claim}
For every $f_k\in\Omega_k$ and $t>0$,
\[
u_k(t)=e^{-tL_k}f_k\in H^1(SM)\cap\Omega_k.
\]
\end{claim}

\begin{proof}
Since $L_k=\nabla_k^*\nabla_k$ is an elliptic connection
Laplacian on the closed surface $M$, the standard
regularization property of its heat operator gives
square-integrable first derivatives in the base variables
for every positive time; see, for instance,
\cite[Section~7.3, p.~100]{ArendtHeat}.
Thus $Xu_k(t)$ and $Hu_k(t)$ belong to $L^2(SM)$.
Moreover, $u_k(t)\in\Omega_k$, so
\[
Vu_k(t)=ik\,u_k(t)\in L^2(SM).
\]
Since $X,H,V$ form a frame of $T(SM)$, all first derivatives
of $u_k(t)$ are square-integrable. Hence $u_k(t)\in H^1(SM)$.
\end{proof}

We use the following consequence of the spectral theorem
\cite[Theorem~1.4.8]{ArendtHeat} and elliptic regularity
\cite[Theorems~4.27--4.28]{HunterPDE}.

\begin{theorem}[Spectral decomposition] \label{thm:spectraldecomposition}
Let $M$ be a closed oriented surface with a
$C^{3+\mathrm{Lip}}$ metric. For each fixed $k\in\mathbb Z$,
the operator $L_k$ has an orthonormal basis
$(\phi_\ell)_{\ell\geq1}$ of $\Omega_k$, with
$\phi_\ell\in H^3(SM)\cap\Omega_k$, satisfying
\[
L_k\phi_\ell=\lambda_\ell\phi_\ell,
\qquad
0\leq\lambda_1\leq\lambda_2\leq\cdots,
\qquad
\lambda_\ell\longrightarrow\infty.
\]
Eigenvalues are repeated according to their multiplicities.
For every $f_k\in\Omega_k$,
\[
f_k=\sum_{\ell=1}^{\infty}
\langle f_k,\phi_\ell\rangle\phi_\ell
\quad\text{in }L^2(SM),
\qquad
\|f_k\|_2^2
=\sum_{\ell=1}^{\infty}
|\langle f_k,\phi_\ell\rangle|^2.
\]
\end{theorem}

\subsection{Spectral heat regularization and its uniform estimate}\label{subsec:spectral}

Fix $k$. Theorem~\ref{thm:spectraldecomposition} supplies an orthonormal eigenbasis
$(\phi_\ell)_\ell$ of $\Omega_k$, with
$L_k\phi_\ell=\lambda_\ell\phi_\ell$ and $\lambda_\ell\ge0$.
The eigenbasis and eigenvalues here depend on $k$. Since $L_k\phi_\ell=\lambda_\ell\phi_\ell$, the heat equation gives
$P_t^k\phi_\ell=e^{-t\lambda_\ell}\phi_\ell$.
By linearity and continuity in the $L^2$ norm,
\[
 P_t^kf_k=e^{-tL_k}f_k
 =\sum_\ell e^{-t\lambda_\ell}
       \langle f_k,\phi_\ell\rangle\phi_\ell,
 \qquad t\ge0,
\]
with convergence in $L^2(SM)$.
 Since $0\le e^{-t\lambda_\ell}\le1$,
it gives $\|P_t^kf_k\|_2\le\|f_k\|_2$. Also
\[
 \|P_t^kf_k-f_k\|_2^2
 =\sum_\ell|e^{-t\lambda_\ell}-1|^2
               |\langle f_k,\phi_\ell\rangle|^2\longrightarrow0
 \qquad(t\downarrow0),
\]
by dominated convergence for the summable Fourier coefficients in this
eigenbasis.

For $t>0$, the same expansion yields
\begin{align*}
 \|\nabla_kP_t^kf_k\|_2^2
 &=\sum_\ell\lambda_\ell e^{-2t\lambda_\ell}
                 |\langle f_k,\phi_\ell\rangle|^2\\
 &\le\left(\sup_{\lambda\ge0}\lambda e^{-2t\lambda}\right)
                   \|f_k\|_2^2
 \lesssim t^{-1}\|f_k\|_2^2.
\end{align*}
Together with \eqref{proof:energy}, this proves
\begin{equation}\label{proof:spectral}
 \|\nabla_kP_t^kf_k\|_2\lesssim t^{-1/2}\|f_k\|_2,
 \qquad
 \|\eta_\pm P_t^kf_k\|_2\lesssim t^{-1/2}\|f_k\|_2.
\end{equation}
The constants are independent of $k$: the only bound used was the
scalar supremum over all $\lambda\ge0$. Thus
$P_t^kf_k\in H^1(SM)\cap\Omega_k$, and the energy identity applies
at every positive heat time. \par

We have now established the upper bound
\begin{equation}\label{eq:upper.bound}
 \|\eta_+u_{2j}(t)\|_2
 \lesssim t^{-1/2}\|f_{2j}\|_2,
 \qquad j\ge1.
\end{equation}
To see what remains to be estimated, recall the recurrence error
$r_\ell$ from \eqref{proof:error-definition}. The Guillemin--Kazhdan
inequality and the triangle inequality give
\[
 \|\eta_+u_{2\ell}(t)\|_2
 \le\|\eta_+u_{2\ell+2}(t)\|_2+\|r_\ell(t)\|_2.
\]
Applying this successively at the same heat time, and then using
\eqref{eq:upper.bound}, yields
\[
 \|\eta_+u_0(t)\|_2
 \lesssim t^{-1/2}\|f_{2j}\|_2
       +\sum_{\ell=0}^{j-1}\|r_\ell(t)\|_2.
\]
We need to choose heat times $t_j$ so that both terms become
small together. There is a trade-off: shortening the heat time
helps reduce the accumulated error, but increases the factor
$t_j^{-1/2}$ multiplying $\|f_{2j}\|_2$. We must therefore leave
enough time for heat to smooth the function, while keeping the
error under control.

The bound we will prove for the accumulated error is
\[
 \sum_{\ell=0}^{j-1}\|r_\ell(t)\|_2
 \lesssim (jt)^{3/2}\|f\|_2.
\]
This suggests choosing
\[
 t_j=\frac{\tau}{j},
\]
for a sufficiently small and fixed $\tau>0$.
Indeed, this makes the error bound $\tau^{3/2}\|f\|_2$, which we
can make as small as we wish by choosing $\tau$ small, independently
of the length of the chain. But we must still check what this
choice does to the other term. It becomes
\[
 t_j^{-1/2}\|f_{2j}\|_2
 =\frac{1}{\sqrt \tau}\sqrt j\,\|f_{2j}\|_2.
\]
For each fixed $\tau>0$, this tends to zero along a subsequence.
To see why, suppose that $\sqrt j\,\|f_{2j}\|_2$ stayed bounded
away from zero for all sufficiently large $j$. Then
$\|f_{2j}\|_2^2$ would be bounded below by a positive multiple
of $1/j$, contradicting the summability of the squared Fourier
coefficients given by Parseval's identity.

Thus $t_j=\tau/j$ allows us to meet both requirements: along a
subsequence, the first term tends to zero, while the accumulated
error remains bounded by a quantity that tends to zero with $\tau$.
This is why we seek the error estimate above. The next two
sections prove it. In Section~\ref{proof:completion-section},
we then pass to the limit in $H^{-1}$ along this subsequence
with $\tau$ fixed, and finally let $\tau\downarrow0$, obtaining
$\|\eta_+f_0\|_{H^{-1}}=0$.

\section{The heat commutator}\label{proof:sec:commutator}

To carry out the two limits in the strategy of Section~\ref{sec:strategy},
we need the accumulated recurrence error at the heat times $t=\tau/j$
to tend to zero with $\tau$, uniformly in the number $j$ of steps:
\[
 \lim_{\tau\downarrow0}\sup_{j\ge1}
 \sum_{\ell=0}^{j-1}\|r_\ell(\tau/j)\|_2=0.
\]
Knowing only that each error tends to zero at a fixed degree does not
give this control, because the number of errors also grows with $j$.

To bound the error $r_\ell$ in \eqref{proof:error-definition}, we will differentiate it with respect
to the heat time $t$. The resulting equation involves the
commutators $[L,\eta_\pm]$. The reason to pass to these
commutators is that we know an explicit formula for $L$,
whereas $P_t$ is obtained by solving the heat equation. We can therefore compute $[L,\eta_\pm]$ directly.

\subsection{The evolution of the recurrence error}
\label{subsec:evolution.error}
Before applying heat, the two terms in the Fourier recurrence
\eqref{eq:recurrence} cancel, so $r_\ell(0)=0$ in distributions.
We want to understand how heat spoils this cancellation.
For this, we follow how $r_\ell(t)$ changes with time, using
the heat equation $\partial_tu_k=-L_ku_k$.

Since $\eta_+$ raises the Fourier degree by one and $\eta_-$
lowers it by one, the two terms
$\eta_+u_{2\ell}$ and $\eta_-u_{2\ell+2}$ both have degree
$2\ell+1$. We can therefore collect their heat evolution into
a single term involving $L_{2\ell+1}r_\ell$.
Using $-\eta_\pm L=-L\eta_\pm+[L,\eta_\pm]$, we obtain,
for $t>0$,
\begin{align*}
 \partial_t r_\ell
 &=-\eta_+L_{2\ell}u_{2\ell}
   -\eta_-L_{2\ell+2}u_{2\ell+2}\\
 &=-L_{2\ell+1}r_\ell
   +[L,\eta_+]u_{2\ell}
   +[L,\eta_-]u_{2\ell+2}.
\end{align*}

If $L$ commuted with $\eta_+$ and $\eta_-$, then $r_\ell$
would satisfy the heat equation with zero initial condition
and would remain zero. This means that the two terms
$\eta_+u_{2\ell}$ and $\eta_-u_{2\ell+2}$ would continue to
cancel after applying heat. To estimate how far their sum
can depart from zero, we therefore estimate the commutator
terms in the equation above.

To estimate $r_\ell$, we will use Duhamel's formula \eqref{proof:Duhamel} to integrate
these contributions, allowing each contribution to evolve
under heat for the remaining time. First, we compute the
commutators $[L,\eta_\pm]$ and bound them in terms of curvature. These bounds will show that the integral in Duhamel's formula
converges in $L^2$, including near time zero, where derivatives
of the heated function may become unbounded.

\subsection{The curvature terms in the commutator}

We want to use the fact that each $f_k$ is supported in $F$,
where both curvature and its gradient vanish. The estimates
\eqref{proof:curvature} tell us how small these quantities remain
near $F$. A bound for the commutator in terms of $K$ and $\nabla K$
would let us use this information: once we prove that heat stays
concentrated near $F$ for short times, we can obtain a sharper
bound for the error.

Curvature enters through the Guillemin--Kazhdan relation
$[\eta_+,\eta_-]=\frac i2KV$ \cite{GK1}. Together with
\eqref{proof:L}, it gives
\[
 [L,\eta_+]=i\eta_+KV+iKV\eta_+,
 \qquad
 [L,\eta_-]=-i\eta_-KV-iKV\eta_-.
\]
Here, on $\Omega_k$, the commutator means
$[L,\eta_\pm]=L_{k\pm1}\eta_\pm-\eta_\pm L_k$.
For $u\in\Omega_k$, angular differentiation multiplies $u$
by $ik$ and $\eta_\pm u$ by $i(k\pm1)$. Thus
\begin{align*}
 [L,\eta_+]u&=-k\eta_+(Ku)-(k+1)K\eta_+u,\\
 [L,\eta_-]u&=k\eta_-(Ku)+(k-1)K\eta_-u.
\end{align*}
Applying the product rule gives
\begin{align}
 [L,\eta_+]u
 &=-(2k+1)K\eta_+u-k(\eta_+K)u,
 \label{proof:commplus}\\
 [L,\eta_-]u
 &=(2k-1)K\eta_-u+k(\eta_-K)u.
 \label{proof:commminus}
\end{align}

These formulas give the dependence on curvature that we were
looking for. Curvature multiplies a first derivative of $u$,
and a first derivative of curvature multiplies $u$ itself.
Since $\eta_\pm$ are combinations of the two horizontal
derivatives, taking norms gives
\begin{equation}\label{proof:commbound}
 \|[L,\eta_\pm]u\|_2
 \lesssim(1+|k|)
 \bigl(\|K\nabla_k u\|_2+\|(\nabla K)u\|_2\bigr).
\end{equation}
The implicit constant is independent of $k$.

By density, \eqref{proof:commplus}--\eqref{proof:commminus}
hold in distributions for every $u\in H^1(SM)\cap\Omega_k$.
Their right-hand sides belong to $L^2$, so
\eqref{proof:commbound} applies to every such $u$.

\subsection{Using the common support}

Every $f_k$ is supported over $F$, where $K$ and $\nabla K$
vanish. This gives us more information than their boundedness:
the curvature bounds \eqref{proof:curvature} tell us that
\[
 |K|\lesssim d_F^2,\qquad |\nabla K|\lesssim d_F.
\]
Applying \eqref{proof:commbound} to $u=P_s^kf_k$ therefore gives
\[
 \|[L,\eta_\pm]P_s^kf_k\|_2
 \lesssim(1+|k|)
 \bigl(\|d_F^2\nabla_kP_s^kf_k\|_2
       +\|d_FP_s^kf_k\|_2\bigr).
\]
The two distance factors come directly from the curvature
terms: $d_F^2$ bounds the curvature multiplying the derivative,
and $d_F$ bounds the curvature gradient multiplying the
function itself.

To make this estimate useful, we need to know how far heat
spreads the initial condition away from $F$. Near $F$, the
distance factors are small. Far from $F$, we need the heated
function and its derivative to be small enough to compensate
for those factors. This is why the ordinary heat derivative
estimate is insufficient: we also need to control where the
function and its derivative are concentrated.

The next section, Section \ref{proof:weighted-section}, will prove the two bounds
\[
 \|d_F^2\nabla_kP_s^kf_k\|_2
 \lesssim \sqrt{s}\,\|f_k\|_2,
 \qquad
 \|d_FP_s^kf_k\|_2
 \lesssim \sqrt{s}\,\|f_k\|_2,
\]
with constants independent of $k$. Substituting these bounds
above will give
\[
 \|[L,\eta_\pm]P_s^kf_k\|_2
 \lesssim (1+|k|)\sqrt{s}\,\|f_k\|_2.
\]
Before proving them, we explain how Duhamel's formula expresses
the errors $r_\ell$ as time integrals of the commutator terms.
This will show how the estimates we seek control the errors
introduced by heat.

\subsection{Integrating the evolution equation}

We first establish the integral identity for each Fourier component.

\begin{claim}[Duhamel's formula]
For every $k\in\mathbb Z$, $f_k\in\Omega_k$, and $t>0$,
the following identity holds in $L^2$:
\begin{equation}\label{proof:Duhamel}
 \eta_\pm P_t^kf_k-P_t^{k\pm1}\eta_\pm f_k
 =\int_0^t P_{t-s}^{k\pm1}
          [L,\eta_\pm]P_s^kf_k\,ds.
\end{equation}
Here $\eta_\pm f_k$ may be a distribution; its heat
evolution is defined by duality:
\begin{equation}\label{eq:duality}
 \langle P_t^{k\pm1}\eta_\pm f_k,\varphi\rangle
 =-\langle f_k,\eta_\mp P_t^{k\pm1}\varphi\rangle,
 \qquad \varphi\in\Omega_{k\pm1}.
\end{equation}
\end{claim}

\begin{proof}
For finite sums of $L_k$-eigenfunctions,
\eqref{proof:Duhamel} follows from Duhamel's formula
(see, for instance,
\cite[Chapter~VI, formula~(7.1), p.~437]{EngelNagel}).

By \eqref{proof:spectral} and duality \eqref{eq:duality},
\[
 \|\eta_\pm P_t^kf_k\|_2
 +\|P_t^{k\pm1}\eta_\pm f_k\|_2
 \lesssim t^{-1/2}\|f_k\|_2.
\]
Since $K$ and $\nabla K$ are bounded,
\eqref{proof:commbound} and \eqref{proof:spectral} give
\begin{equation}\label{eq:estimate}
 \|[L,\eta_\pm]P_s^kf_k\|_2
 \lesssim(1+|k|)(1+s^{-1/2})\|f_k\|_2.
\end{equation}
The heat operators have norm bounded by one, and
$1+s^{-1/2}$ is integrable on $(0,t)$.
Thus every term in \eqref{proof:Duhamel} depends
continuously on $f_k$ in $L^2$, and the identity
extends by density.
\end{proof}

We use Duhamel's formula \eqref{proof:Duhamel} to write
$r_\ell(t)$ in terms of the commutators $[L,\eta_\pm]$.
Starting from its definition, we add and subtract
$P_t^{2\ell+1}\eta_+f_{2\ell}$ and
$P_t^{2\ell+1}\eta_-f_{2\ell+2}$:
\begin{align*}
 r_\ell(t)
 &=\eta_+P_t^{2\ell}f_{2\ell}
   +\eta_-P_t^{2\ell+2}f_{2\ell+2}\\
 &=\bigl(\eta_+P_t^{2\ell}f_{2\ell}
         -P_t^{2\ell+1}\eta_+f_{2\ell}\bigr)\\
 &\quad+\bigl(\eta_-P_t^{2\ell+2}f_{2\ell+2}
         -P_t^{2\ell+1}\eta_-f_{2\ell+2}\bigr)\\
 &\quad+P_t^{2\ell+1}
       \bigl(\eta_+f_{2\ell}+\eta_-f_{2\ell+2}\bigr).
\end{align*}
The last line vanishes by \eqref{proof:recurrence},
which follows from $Xf=0$. Each of the two remaining
parentheses has exactly the form of the left-hand side
of \eqref{proof:Duhamel}. Applying that formula gives
\[
 r_\ell(t)=\int_0^t P_{t-s}^{2\ell+1}
 \bigl([L,\eta_+]P_s^{2\ell}f_{2\ell}
       +[L,\eta_-]P_s^{2\ell+2}f_{2\ell+2}\bigr)\,ds.
\]
Taking norms and using that the operator norm of
$P_{t-s}^{2\ell+1}$ on $L^2$ is bounded by one, we obtain
\begin{equation}\label{eq:r.ell.in.terms.of.L.eta}
 \|r_\ell(t)\|_2
 \le\int_0^t
 \bigl(\|[L,\eta_+]P_s^{2\ell}f_{2\ell}\|_2
       +\|[L,\eta_-]P_s^{2\ell+2}f_{2\ell+2}\|_2\bigr)\,ds.
\end{equation}
Substituting \eqref{eq:estimate} into
\eqref{eq:r.ell.in.terms.of.L.eta} and integrating in $s$ yields
\[
 \|r_\ell(t)\|_2
 \lesssim(\ell+1)(t+\sqrt t)
 \bigl(\|f_{2\ell}\|_2+\|f_{2\ell+2}\|_2\bigr).
\]
Cauchy--Schwarz and Parseval then give
\[
 \sum_{\ell=0}^{j-1}\|r_\ell(t)\|_2
 \lesssim(t+\sqrt t)j^{3/2}\|f\|_2.
\]

At the times $t=\tau/j$, with $\tau>0$ fixed, the bound becomes
\[
\sum_{\ell=0}^{j-1}\|r_\ell(t)\|_2
 \lesssim (\tau\sqrt j+\sqrt{\tau}\,j)\|f\|_2.
\]
This bound was good for the purpose of justifying Duhamel's formula \eqref{proof:Duhamel}, but it is too weak for our purpose. It grows with $j$, however small we choose $\tau$.
The estimate therefore does not show that the sum
$\sum_{\ell=0}^{j-1}\|r_\ell(\tau/j)\|_2$ remains small
as $j$ increases, which is what we need.

This is because the bound obtained so far used only that $K$ and
$\nabla K$ are bounded. To improve it, we use that
both vanish on $F$, where all the initial modes
$f_k$ are supported.

\section{Weighted heat estimates}\label{proof:weighted-section}

To bound $r_\ell$, we need to estimate the commutators $[L,\eta_\pm]$
appearing in its integral expression \eqref{eq:r.ell.in.terms.of.L.eta}. By
\eqref{proof:commbound},
\[
 \|[L,\eta_\pm]P_t^kf_k\|_2
 \lesssim(1+|k|)
 \bigl(\|K\nabla_kP_t^kf_k\|_2
       +\|(\nabla K)P_t^kf_k\|_2\bigr).
\]
The idea is that the factors in each product are small
in complementary regions. Near $F$, both $K$ and
$\nabla K$ are small because they vanish on $F$.
Farther from $F$, where these coefficients may be
larger, we will show that the heated function and its
derivative have little mass. Thus each product is
small: near $F$ because of the curvature, and away
from $F$ because of the heat evolution. To show this
concentration near $F$, we use Kannai's formula
\eqref{eq:kannai}, which expresses heat as a Gaussian
average of waves whose support we can control.

In this section we prove that heat applied to a function
supported over $F$ remains concentrated near $F$, in the
following precise sense:
\[
 \|d_FP_t^kf_k\|_2\lesssim\sqrt t\,\|f_k\|_2,
 \qquad
 \|d_F^2\nabla_kP_t^kf_k\|_2
 \lesssim\sqrt t\,\|f_k\|_2,
\]
with constants independent of $k$. Here $d_F(x)$ is the
distance from $x$ to $F$, and $f_k\in\Omega_k$ is supported
over $F$. The distance factors give greater weight to
points farther from $F$, so these estimates control how
much of the heated function and its derivative can lie
far from their initial support.

These are the estimates we need to bound the errors
introduced by heat regularization. Since
\[
 |K|\lesssim d_F^2,\qquad |\nabla K|\lesssim d_F,
\]
the two weighted estimates above will give
\[
 \|[L,\eta_\pm]P_t^kf_k\|_2
 \lesssim(1+|k|)\sqrt t\,\|f_k\|_2.
\]

As we stated above, we use Kannai's formula \eqref{eq:kannai}, which expresses
heat as a Gaussian average of solutions of the wave
equation. Waves have a useful property that heat does
not: a wave starting over $F$ is supported where
$d_F\le r$ at time $r$. We first prove this support
property. Within that region, $d_F\le r$ and
$d_F^2\le r^2$, so the distance factors can be bounded
by powers of $r$. Combining this observation with the
wave energy estimate and Kannai's formula \eqref{eq:kannai} will give
the two weighted heat estimates.

Throughout this section, $f_k\in\Omega_k$ is an $L^2$ initial
condition supported over $F$. Supports are understood on the
base $M$, up to sets of measure zero. The estimates use the
connection Laplacian $L_k=\nabla_k^*\nabla_k$ and will be
independent of $k$. We use $k$ to denote an arbitrary Fourier degree. We will
apply these estimates with $k=2\ell$ and $k=2\ell+2$
to bound each error $r_\ell$.

\subsection{The wave solution and its energy}

Fix $k\in{\mathbb Z}$ and consider the wave equation with the same horizontal Laplacian $L_k$ as the heat equation \eqref{eq:heat}:
\begin{equation*}
 w_{rr}+L_kw=0,\qquad w(0)=0,\qquad w_r(0)=f_k.
\end{equation*}
Here $r$ is wave time; $t$ continues to denote heat time.
In an orthonormal eigenbasis $(\phi_\ell)_\ell$ for $L_k$,
the solution and its time derivative are
\begin{align*}
 w(r)&=\sum_\ell
 \frac{\sin(r\sqrt{\lambda_\ell})}{\sqrt{\lambda_\ell}}
 \langle f_k,\phi_\ell\rangle\phi_\ell,\\
 z(r):=w_r(r)&=\sum_\ell\cos(r\sqrt{\lambda_\ell})
 \langle f_k,\phi_\ell\rangle\phi_\ell.
\end{align*}
When $\lambda_\ell=0$, we interpret
$\sin(r\sqrt{\lambda_\ell})/\sqrt{\lambda_\ell}$ as $r$,
its limit as $\lambda_\ell\to0$.
The series show that $w$ is continuously differentiable in
$L^2$, with derivative $z$, and that
$\|w(r)\|_2\le r\|f_k\|_2$ for $r\ge0$.
The energy identity \eqref{proof:energy} and Parseval give
\begin{align*}
 \|\nabla_kw(r)\|_2^2
 &=\sum_{\lambda_\ell>0}
 \lambda_\ell
 \left|\frac{\sin(r\sqrt{\lambda_\ell})}
 {\sqrt{\lambda_\ell}}\right|^2
 |\langle f_k,\phi_\ell\rangle|^2\\
 &=\sum_\ell\sin^2(r\sqrt{\lambda_\ell})
 |\langle f_k,\phi_\ell\rangle|^2.
\end{align*}
Consequently $w(r)\in H^1\cap\Omega_k$, its covariant
derivative is continuous in $L^2$, and
\begin{equation}\label{proof:waveenergy}
 \|z(r)\|_2^2+\|\nabla_kw(r)\|_2^2=\|f_k\|_2^2.
\end{equation}
The last identity follows by adding the squared sine and
cosine multipliers for each eigencomponent. 

\subsection{Finite propagation from a closed support}
The purpose of this subsection is to prove that a wave
starting over $F$ vanishes at every point whose distance
from $F$ exceeds the elapsed wave time $r$. Thus its
support at time $r$ is contained in
\[
 F_r=\{x\in M:d_F(x)\le r\}.
\]
Although $M$ is compact, this gives useful information
for small $r$: the wave remains close to its initial
support. This is precisely what we need in Kannai's
formula \eqref{eq:kannai}. At small heat time $t$, the
factor $e^{-r^2/(4t)}$ gives little weight to wave times
much larger than $\sqrt t$, while waves at shorter
times remain near $F$. We will combine this support
property with the wave energy estimate to bound the
heated function and its derivative away from $F$.
\begin{lemma}\label{proof:propagation-lemma}
For every $r\ge0$,
\begin{equation}\label{proof:propagation}
 \operatorname{supp}w(r),\quad
 \operatorname{supp}z(r),\quad
 \operatorname{supp}\nabla_kw(r)
 \subset F_r:=\{x\in M:d_F(x)\le r\}.
\end{equation}
\end{lemma}

\begin{proof}
The triangle inequality makes $d_F$ a $1$-Lipschitz function,
so $|\nabla d_F|\le1$ almost everywhere. We can therefore use
the distance itself without assuming any regularity of $F$
or of its boundary.

To prove that the wave remains within distance $r$
of $F$, we measure its energy outside that region.
We do this by multiplying the energy density by a
nonnegative cutoff that vanishes where $d_F\le r$.
We denote the resulting integral by $E(r)$, the weighted energy:

\[
 E(r)=\frac12\int_M(d_F-r)_+
       \bigl(|z(r)|^2+|\nabla_kw(r)|^2\bigr)\,d\bar a,
 \qquad (d_F-r)_+=\max\{d_F-r,0\}.
\]
The weight is positive precisely outside $F_r$.
We will prove that $E(r)\le E(0)$; since the initial
condition is supported over $F$, its initial weighted
energy $E(0)$ is zero.

First let the initial condition be a finite sum of
$L_k$-eigenfunctions, without requiring it to be supported
over $F$. In this case $w_{rr}$ and $\nabla_kw_r$ are
continuous in $L^2$, and $E$ is absolutely continuous.
For almost every $(r,x)$, the time derivative of the
weight is $-\mathbf1_{\{d_F>r\}}$ and its weak spatial
gradient is $\mathbf1_{\{d_F>r\}}\nabla d_F$. 

The Fourier modes, and hence $w$, may be complex-valued.
When differentiating a squared absolute value, we therefore
differentiate both the function and its complex conjugate:
\[
 \partial_r|w_r|^2
 =w_{rr}\overline{w_r}+w_r\overline{w_{rr}}
 =2\operatorname{Re}(w_{rr}\overline{w_r}).
\]
The same rule applies to $|\nabla_kw|^2$, which explains
the real parts in the calculation below.

We want to prove that $E'(r)\leq 0$, so we will compute $E'(r)$. Differentiating the weight, $|z|^2=|w_r|^2$, and
$|\nabla_kw|^2$ gives, for almost every $r$,
\begin{align*}
 E'(r)
 ={}&-\frac12\int_{\{d_F>r\}}
       \bigl(|z|^2+|\nabla_kw|^2\bigr)\,d\bar a\\
 &+\operatorname{Re}\int_M(d_F-r)_+
       \langle w_{rr},z\rangle\,d\bar a\\
 &+\operatorname{Re}\int_M(d_F-r)_+
       \langle\nabla_kw,\nabla_kz\rangle\,d\bar a.
\end{align*}
Here we used $w_r=z$ and
$\partial_r(\nabla_kw)=\nabla_kz$.

In the second integral, we substitute the wave equation
$w_{rr}=-\nabla_k^*\nabla_kw$ and integrate by parts:
\begin{align*}
 &\operatorname{Re}\int_M(d_F-r)_+
       \langle w_{rr},z\rangle\,d\bar a\\
 &\quad=-\operatorname{Re}\int_M
       \langle\nabla_kw,\nabla_k((d_F-r)_+z)\rangle
       \,d\bar a\\
 &\quad=-\operatorname{Re}\int_M(d_F-r)_+
       \langle\nabla_kw,\nabla_kz\rangle\,d\bar a\\
 &\qquad\quad-\operatorname{Re}\int_{\{d_F>r\}}
       \langle\nabla_kw(\nabla d_F),z\rangle\,d\bar a.
\end{align*}
Integration by parts is valid because $(d_F-r)_+z\in H^1$.
The last two integrals come from differentiating $z$
and the weight, respectively.

The integral containing $\nabla_kz$ is exactly the
negative of the last integral in the expression for
$E'(r)$. After cancellation, we obtain
\begin{align*}
 E'(r)
 ={}&-\frac12\int_{\{d_F>r\}}
       \bigl(|z|^2+|\nabla_kw|^2\bigr)\,d\bar a\\
 &-\operatorname{Re}\int_{\{d_F>r\}}
       \langle\nabla_kw(\nabla d_F),z\rangle\,d\bar a.
\end{align*}
Since $|\nabla d_F|\le1$, Cauchy--Schwarz gives
\[
 -\operatorname{Re}
 \langle\nabla_kw(\nabla d_F),z\rangle
 \le|\nabla_kw|\,|z|.
\]
Consequently,
\begin{align*}
 E'(r)
 &\le-\frac12\int_{\{d_F>r\}}
       \bigl(|z|^2+|\nabla_kw|^2
             -2|\nabla_kw|\,|z|\bigr)\,d\bar a\\
 &=-\frac12\int_{\{d_F>r\}}
       \bigl(|z|-|\nabla_kw|\bigr)^2\,d\bar a
 \le0.
\end{align*}
Hence $E(r)\le E(0)$ for all $r\ge0$.

Since the weight $(d_F-r)_+$ is bounded, the energy
identity \eqref{proof:waveenergy} makes $E(r)$ continuous
with respect to the initial condition in $L^2$.
Hence $E(r)\le E(0)$ extends to all initial conditions
in $\Omega_k$ by density.

For the original initial condition,
\[
 E(0)=\frac12\int_Md_F|f_k|^2\,d\bar a=0.
\]
Nonnegativity forces $E(r)=0$, so $z(r)$ and
$\nabla_kw(r)$ vanish outside $F_r$. Finally,
\[
 w(r)=\int_0^r z(s)\,ds\quad\text{in }L^2.
\]
Each integrand is supported in $F_s\subset F_r$, which
proves the support assertion for $w(r)$ as well.
\end{proof}

\subsection{The two weighted heat estimates}

We want to show that, at short heat times, most of the
mass of the heated function and its derivative remains
near $F$. We introduced waves because we can say exactly
how far they have traveled: at time $r$, they vanish
beyond distance $r$ from $F$. Kannai's formula
\eqref{eq:kannai} connects this information to heat.
It expresses the heated function as an integral of
waves, with the Gaussian weight $e^{-r^2/(4t)}$.
To contribute at a point at distance $d_F(x)$ from $F$,
a wave must have traveled for at least that long.
When $d_F(x)$ is large compared with $\sqrt t$, all
such wave times receive a small Gaussian weight.
This is why the formula is useful here: it turns the
support restriction for waves into an estimate of
how much heat spreads away from $F$. Together with
the wave energy bound, it will give the estimates
for the heated function and its derivative that we need.

\begin{proposition}\label{proof:weighted-prop}
For every $f_k\in\Omega_k$ supported over $F$ and every $t>0$,
\begin{equation}\label{proof:weighted}
 \|d_FP_t^kf_k\|_2\lesssim t^{1/2}\|f_k\|_2,
 \qquad
 \|d_F^2\nabla_kP_t^kf_k\|_2\lesssim t^{1/2}\|f_k\|_2.
\end{equation}
\end{proposition}

\begin{proof}
The heat--wave formula of Kannai \cite{Kannai} gives
\begin{equation}\label{proof:HWT}
 P_t^kf_k=\frac1{\sqrt{\pi t}}
       \int_0^\infty e^{-r^2/(4t)}z(r)\,dr.
\end{equation}
In our setting, it follows by applying the scalar identity
\[
 \frac1{\sqrt{\pi t}}\int_0^\infty
 e^{-r^2/(4t)}\cos(r\sqrt\lambda)\,dr=e^{-t\lambda},
 \qquad\lambda\ge0,
\]
to each eigencomponent. The integral converges in $L^2$
because $\|z(r)\|_2\le\|f_k\|_2$ and the Gaussian weight
is integrable.

By \eqref{proof:propagation}, $d_F\le r$ on the support
of $z(r)$. Hence \eqref{proof:waveenergy} gives
$\|d_Fz(r)\|_2\le r\|f_k\|_2$. Substitution into
\eqref{proof:HWT} yields
\[
 \|d_FP_t^kf_k\|_2
 \le\frac{\|f_k\|_2}{\sqrt{\pi t}}
      \int_0^\infty r e^{-r^2/(4t)}\,dr
 \lesssim t^{1/2}\|f_k\|_2.
\]
The substitution $r=\sqrt t\,s$ gives the final power of $t$.

For the derivative estimate, the energy identity controls
$\nabla_kw$, whereas \eqref{proof:HWT} involves $z=w_r$.
We therefore integrate that formula by parts in wave time:
\begin{equation}\label{proof:HWTsine}
 P_t^kf_k=\frac1{\sqrt{\pi t}}
       \int_0^\infty\frac r{2t}e^{-r^2/(4t)}w(r)\,dr.
\end{equation}
The boundary terms vanish because $w(0)=0$ and
$\|w(r)\|_2\le r\|f_k\|_2$.
Moreover,
\[
 \|w(r)\|_2+\|\nabla_kw(r)\|_2
 \le(r+1)\|f_k\|_2.
\]
Thus the integral and the integral of its covariant
derivative both converge in $L^2$. By closedness of the
weak derivative, we may take $\nabla_k$ inside
\eqref{proof:HWTsine}.

The support and energy bounds give
$\|d_F^2\nabla_kw(r)\|_2\le r^2\|f_k\|_2$.
Consequently,
\[
 \|d_F^2\nabla_kP_t^kf_k\|_2
 \le\frac{\|f_k\|_2}{\sqrt{\pi t}}
       \int_0^\infty\frac{r^3}{2t}e^{-r^2/(4t)}\,dr
 \lesssim t^{1/2}\|f_k\|_2.
\]
Again, the substitution $r=\sqrt t\,s$ gives the stated
power. Only the wave energy, its support, and scalar
Gaussian integrals were used, so the constants are
independent of $k$.
\end{proof}

\subsection{Returning to the commutator and the recurrence errors}

The weighted estimates now give the bound needed at the end
of Section~\ref{proof:sec:commutator}. Indeed,
\eqref{proof:curvature} and \eqref{proof:weighted} yield
\begin{align*}
 \|K\nabla_kP_s^kf_k\|_2
 &\lesssim\|d_F^2\nabla_kP_s^kf_k\|_2
 \lesssim s^{1/2}\|f_k\|_2,\\
 \|(\nabla K)P_s^kf_k\|_2
 &\lesssim\|d_FP_s^kf_k\|_2
 \lesssim s^{1/2}\|f_k\|_2.
\end{align*}
Inserting these bounds into \eqref{proof:commbound} gives
\begin{equation}\label{proof:forcing}
 \|[L,\eta_\pm]P_s^kf_k\|_2
 \lesssim(1+|k|)s^{1/2}\|f_k\|_2.
\end{equation}
The curvature and support assumptions have changed the
general bound $1+s^{-1/2}$ used to justify time integration
into a bound that tends to zero as $s\downarrow0$.
Integrating it gives the following estimate.

\begin{proposition}\label{proof:commutator-prop}
For every $k\in\mathbb Z$, every $f_k\in\Omega_k$ supported
over $F$, and every $t>0$,
\begin{equation*}
 \|\eta_\pm P_t^kf_k-P_t^{k\pm1}\eta_\pm f_k\|_2
 \lesssim(1+|k|)t^{3/2}\|f_k\|_2.
\end{equation*}
The expression
$P_t^{k\pm1}\eta_\pm f_k$ has the $L^2$ meaning established
in Section~\ref{proof:sec:commutator}.
\end{proposition}

\begin{proof}
Use Duhamel's formula \eqref{proof:Duhamel}, the fact that the heat operator
has norm bounded by one on $L^2$, and \eqref{proof:forcing}:
\begin{align*}
 \|\eta_\pm P_t^kf_k-P_t^{k\pm1}\eta_\pm f_k\|_2
 &\le\int_0^t\|[L,\eta_\pm]P_s^kf_k\|_2\,ds\\
 &\lesssim(1+|k|)\|f_k\|_2\int_0^t s^{1/2}\,ds\\
 &\lesssim(1+|k|)t^{3/2}\|f_k\|_2.
\end{align*}
\end{proof}

We have now obtained both parts of the information needed
to control $r_\ell$: the factor $t^{3/2}$ describes the
dependence on heat time, and $1+|k|$ describes the dependence
on Fourier degree. The next section applies this estimate
to the two modes defining each error, sums the errors along
the chain, and obtains a bound that tends to zero with $\tau$
at times $t=\tau/j$, uniformly in $j$. This is the control
required at the beginning of Section~\ref{proof:sec:commutator}.
The last step is to use this control in the $H^{-1}$ estimate for
$\eta_+f_0$.

\section{Completion of the proof}\label{proof:completion-section}

We now combine the estimates at one fixed heat time, and only afterwards
choose that time as a function of the last Fourier degree in the chain.
As before, put \(u_k(t)=P_t^kf_k\). Every \(u_k(t)\), for \(t>0\),
belongs to \(H^1\cap\Omega_k\), so the Guillemin--Kazhdan identity
\eqref{proof:GK} applies to it. We now prove the analytic target
\(\|\eta_+f_0\|_{H^{-1}}=0\) from Section~\ref{sec:strategy},
and then conclude that \(f_0=f=0\) almost everywhere.

\subsection{The failure of invariance after regularization}

The operators \(P_t^k\) define an operator on \(L^2(SM)\) with
operator norm bounded by one:
\[
 P_tf=\sum_{k\in\mathbb Z}P_t^kf_k=\sum_{k\in\mathbb Z}u_k(t).
\]
Indeed, this series converges in \(L^2\) by Parseval and the inequality
\(\|P_t^kf_k\|_2\le\|f_k\|_2\). For \(t>0\), the spectral
estimate \eqref{proof:spectral} also gives
\[
 \sum_{k\in\mathbb Z}\|\eta_\pm u_k(t)\|_2^2
 \lesssim t^{-1}\sum_{k\in\mathbb Z}\|f_k\|_2^2<\infty.
\]
For each sign, the terms $\eta_\pm u_k(t)$ have distinct
Fourier degrees and are therefore orthogonal. The two series
converge in $L^2$, and their sum represents $XP_tf$ in
distributions.

Although \(Xf=0\), the function \(XP_tf=[X,P_t]f\) need not vanish.
Its degree-\((2\ell+1)\) component is
\[
 r_\ell(t):=\bigl([X,P_t]f\bigr)_{2\ell+1}
 =\eta_+u_{2\ell}(t)+\eta_-u_{2\ell+2}(t),
 \qquad \ell\ge0.
\]
Thus \(r_\ell(t)\) is exactly the error in the Fourier recurrence
after applying heat. Apply \(P_t^{2\ell+1}\) to the distributional
identity \eqref{proof:recurrence}. Both terms now lie in \(L^2\),
with the interpretation established in Section~\ref{proof:sec:commutator},
and their sum is zero. Subtracting this equality from
\eqref{proof:error-definition} gives
\[
 r_\ell(t)=[\eta_+,P_t]f_{2\ell}
             +[\eta_-,P_t]f_{2\ell+2}.
\]
Every \(f_k\) is supported over the same set \(F\). Therefore
Proposition~\ref{proof:commutator-prop} yields
\begin{equation}\label{proof:error}
 \|r_\ell(t)\|_2
 \lesssim t^{3/2}(\ell+1)
       \bigl(\|f_{2\ell}\|_2+\|f_{2\ell+2}\|_2\bigr).
\end{equation}

\subsection{An upper bound for the degree-zero derivative}\label{subsec:upper.bound}

Fix \(t>0\). The recurrence error and the Guillemin--Kazhdan
inequality \eqref{proof:monotone}, applied in degree \(2\ell+2\), give
\[
 \|\eta_+u_{2\ell}(t)\|_2
 \le\|\eta_-u_{2\ell+2}(t)\|_2+\|r_\ell(t)\|_2
 \le\|\eta_+u_{2\ell+2}(t)\|_2+\|r_\ell(t)\|_2.
\]
Apply this successively for \(\ell=0,\ldots,j-1\), keeping the
same heat time in every term. We obtain
\begin{equation}\label{proof:chain}
 \|\eta_+u_0(t)\|_2
 \le\|\eta_+u_{2j}(t)\|_2
       +\sum_{\ell=0}^{j-1}\|r_\ell(t)\|_2.
\end{equation}

By Cauchy--Schwarz and \eqref{proof:error},
\begin{align*}
 \sum_{\ell=0}^{j-1}\|r_\ell(t)\|_2
 &\lesssim t^{3/2}
 \left(\sum_{\ell=0}^{j-1}(\ell+1)^2\right)^{1/2}
 \left(\sum_{\ell=0}^{j-1}
       (\|f_{2\ell}\|_2+\|f_{2\ell+2}\|_2)^2\right)^{1/2}.
\end{align*}
The first sum is at most a fixed multiple of \(j^3\). For the second,
use \((a+b)^2\le2a^2+2b^2\): each coefficient occurs at most twice,
and hence
\[
 \sum_{\ell=0}^{j-1}
       (\|f_{2\ell}\|_2+\|f_{2\ell+2}\|_2)^2
 \le4\sum_{\ell=0}^{j}\|f_{2\ell}\|_2^2
 \le4\|f\|_2^2.
\]
The last inequality is Parseval. We conclude that
\begin{equation}\label{proof:sum}
 \sum_{\ell=0}^{j-1}\|r_\ell(t)\|_2
 \lesssim(tj)^{3/2}\|f\|_2.
\end{equation}
Combining \eqref{proof:chain}, \eqref{proof:sum}, and the
degree-uniform derivative estimate \eqref{proof:spectral} gives
\[
 \|\eta_+P_t^0f_0\|_2
 \lesssim t^{-1/2}\|f_{2j}\|_2+(tj)^{3/2}\|f\|_2,
 \qquad j\ge1.
\]

\subsection{The two limits}\label{subsec:two.limits}

We use the space $H^{-1}(SM)$ defined in
\eqref{eq:negative-norm}. The derivative estimate
\eqref{eq:weak-derivative} and convergence of heat to its initial
condition give
\[
 \|\eta_+(P_t^0f_0-f_0)\|_{H^{-1}}
 \lesssim\|P_t^0f_0-f_0\|_2\longrightarrow0
 \qquad(t\downarrow0).
\]

Fix an arbitrary \(\tau\) with \(0<\tau\le1\), and set \(t=\tau/j\).
The upper bound just proved and the inclusion \(L^2\subset H^{-1}\)
give
\[
 \|\eta_+P_{\tau/j}^0f_0\|_{H^{-1}}
 \lesssim \tau^{-1/2}\sqrt j\,\|f_{2j}\|_2
       +\tau^{3/2}\|f\|_2,
 \qquad j\ge1.
\]
 Parseval implies
\[
 \liminf_{j\to\infty}\sqrt j\,\|f_{2j}\|_2=0.
\]
Indeed, if this lower limit were positive, there would be an
\(\varepsilon>0\) such that
\(\|f_{2j}\|_2^2\ge\varepsilon^2/j\) for all sufficiently large
\(j\). Summing would contradict Parseval because the harmonic series
diverges.

Choose a subsequence along which \(\sqrt j\,\|f_{2j}\|_2\to0\).
With \(\tau\) still fixed, we have \(\tau/j\to0\) on this subsequence,
so convergence in \(H^{-1}\) and the preceding upper bound yield
\[
 \|\eta_+f_0\|_{H^{-1}}\lesssim \tau^{3/2}\|f\|_2.
\]
This holds for every \(0<\tau\le1\), with the same implicit constant.
Only now let \(\tau\downarrow0\). It follows that
\(\eta_+f_0=0\) in \(H^{-1}(SM)\), and hence in distributions.

\subsection{Vanishing of the singular indicator}
\label{the.end}
Since $f_0$ is real, complex conjugation gives
$\eta_-f_0=0$, and hence $Xf_0=Hf_0=0$ in distributions.
As $f_0$ depends only on the base point and $X,H$
project to an orthonormal basis of each tangent plane
of $M$, its distributional gradient vanishes.
Connectedness of $M$ therefore implies that $f_0$
is constant almost everywhere.

The set \(\{K<0\}\) is nonempty and open, and \(f_0=0\) there.
The constant is therefore zero, proving \(f_0=0\), or equivalently
\(\|f_0\|_{H^{-1}}=0\), as announced at the start of the outline.
Since \(f\ge0\) and \(f_0\) is its fiber average, this implies
\(f=0\) almost everywhere. Consequently
\(m(\mathrm{Sing})=0\). The regular set has full Liouville measure,
so Pesin's ergodicity theorem on that set \cite{Pesin77} proves the
theorem stated in the introduction.

\bibliographystyle{plain}
\bibliography{geodesic-ergodicity}
\end{document}